\documentclass{amsart}
\usepackage{stmaryrd}
\usepackage{amsmath,amssymb,hyperref}
\usepackage{xcolor}
\usepackage{graphicx}

\newtheorem{thm}{Theorem}[section]
\newtheorem{lemma}[thm]{Lemma}
\newtheorem{claim}{Claim}[thm]

\newtheorem{theorem}[thm]{Theorem}

\newtheorem{corollary}[thm]{Corollary}
\newtheorem{fact}[thm]{Fact}

\newtheorem*{thma}{Theorem A}
\newtheorem*{corb}{Corollary B}

\newtheorem*{corc}{Corollary C}
\newtheorem*{thmd}{Theorem D}

\newtheorem*{core}{Corollary E}
\newtheorem*{etheorem}{Theorem}
\newtheorem*{econjecture}{Conjecture}
\newtheorem*{equestion}{Question}

\theoremstyle{definition}

\newtheorem{definition}[thm]{Definition}

\newtheorem{remark}[thm]{Remark}

\newcommand*\axiomfont[1]{\textsf{\textup{#1}}}
\newcommand\zfc{\axiomfont{ZFC}}

\def\s{\subseteq}
\newcommand\br{\blacktriangleright}
\newcommand\diagonalintersection{\bigtriangleup}
\DeclareMathOperator\diagonalunion{\bigtriangledown}

\usepackage{enumitem}
\setlist[enumerate,1]{label={(\roman*)}}

\DeclareMathOperator{\dom}{dom}

\DeclareMathOperator{\ord}{Ord}
\newenvironment{why}[1][Proof]{\proof[#1]\mbox{}}{\endproof}

\DeclareMathOperator{\im}{im}

\DeclareMathOperator{\rk}{rk}

\DeclareMathOperator{\p}{\mathcal P}

\DeclareMathOperator{\pom}{{\mathcal P_{\!\omega_1}}}


\title{A Ramsey theorem for the reals}
\date{Preliminary version as of \today}
\author{Tanmay C. Inamdar}
\address{Department of Mathematics, Ben-Gurion University of the Negev, P.O.B. 653, Be’er Sheva, 84105 Israel}
\email{tci.math@protonmail.com}

\begin{document}
	\begin{abstract} 
We prove that for every colouring of pairs of reals with finitely-many colours, there is a set homeomorphic to the rationals which takes no more than two colours. This was conjectured by Galvin in 1970, and a colouring of Sierpi{\'n}ski from 1933 witnesses that the number of colours cannot be reduced to one. Previously in 1985 Shelah had shown that a stronger statement is consistent with a forcing construction assuming the existence of large cardinals. Then in 2018 Raghavan and Todor{\v c}evi{\'c} had proved it assuming the existence of large cardinals. We prove it in $\zfc$. In fact Raghavan and Todor{\v c}evi{\'c} proved, assuming more large cardinals, a similar result for a large class of topological spaces. We prove this also, again in $\zfc$. 
	\end{abstract}
	\maketitle
\section{Introduction}

We prove here the following:
\begin{thma} Let $X$ be an uncountable set of reals. Let $K$ be a positive natural number and let $c: [X]^2 \rightarrow \{0,1,\ldots, K-1\}$ be a colouring of pairs of elements of
$X$ into $K$-many colours. Then for some $Y \s X$ homeomorphic to $\mathbb Q$, $c$ takes at most $2$ colours on pairs of elements of $Y$; that is, $|c``[Y]^2| \leq 2$.\footnote{For a set $Z$, we denote by $[Z]^2$ the collection of all two element subsets of $Z$, and $|Z|$ denotes the cardinality of the set $Z$. 
Additional notation can be found in Section~\ref{sectionnotation}.}
\end{thma}
Theorem~A is the most interesting case of a more general result, Theorem~D, that we will discuss in Section~\ref{subsectiongeneral}. An immediate consequence of Theorem~A is the following Ramsey theorem for the reals.
\begin{corb} For every positive natural number $K$ and $c:[\mathbb R]^2 \rightarrow \{0,1,\ldots, K-1\}$ a colouring of pairs of reals, there is a set of reals $Y$, homeomorphic to $\mathbb Q$, such that $|c``[Y]^2| \leq 2$.
\end{corb}
We now commence on a short survey of infinite Ramsey theory aimed towards placing Theorem~A and Corollary~B in their appropriate context. We will deliberately avoid using the Erd{\H o}s-Hajnal-Rado arrow \cite[\S~18]{MR202613} in order to avoid overloading the reader with notation. Below, by standard set-theoretic convention, an ordinal will also represent the set of its predecessors. In particular, for $K$ a natural number, $K = \{0,1,\ldots, K-1\}$.

We start in 1928 with Ramsey's Theorem.
\begin{etheorem}[Ramsey, \cite{MR1576401}] Let $K$ be a positive natural number. For any colouring $c:[\mathbb N]^2 \rightarrow K$, there is an infinite $Y \s \mathbb N$ such that $|c``[Y]^2| = 1$.
\end{etheorem}
Interpreting the above as a result about the ordertype (as opposed to the size\footnote{For a perspective centred on cardinalities, starting from Erd{\H o}s and Tarski's isolation of the notion of \emph{weakly compact cardinals}, \cite{MR167422}, see \cite{MR1994835}.}) of the natural numbers, our focus here will be attempts to generalise Ramsey's Theorem to the most basic infinite structures.

Sierpi{\'n}ski, answering a question of Knaster, constructed in 1933 the following colouring (\cite{MR1556708}). Let $<^*$ be a well-ordering of $\mathbb R$ and $<$ the usual order on $\mathbb R$. Let $c_{\mathrm s}:[\mathbb R]^2 \rightarrow 2$ be defined via: for $x< y$ from $\mathbb R$, 
$$c_{\mathrm s}(x,y):= \begin{cases} 0 &\text{ if } x<^* y; \\
1&\text{ if } y<^* x.
\end{cases}$$
Then for any uncountable set $Y \s \mathbb R$, $c_{\mathrm s}``[Y]^2 = 2$. This is because there cannot be an order-preserving embedding of the first uncountable ordinal $\omega_1$ or its reverse into $\mathbb R$, since $\mathbb R$ has a countable dense set. It follows that there is no direct analogue of Ramsey's Theorem for $\mathbb R$.\footnote{Curiously, in \cite{MR1556708} Sierpi{\'n}ski seems to be unaware of Ramsey's work, and indeed proves Ramsey's Theorem himself in the paper. That is, the infinite Ramsey theorem for pairs, as stated above. He also mentions that Lindenbaum informed him that he was aware of this fact.}

Examining Sierpi{\'n}ski's colouring $c_{\mathrm s}$, since $<^*$ is a well-ordering, we also have that on any set $Y \s \mathbb R$ order-isomorphic to $\mathbb Z$, $c_{\mathrm s}``[Y]^2 = 2$. Galvin in a letter to Laver, \cite{galvin}, observed that on any non-empty $Y \s \mathbb R$ which is dense-in-itself, that is, has no isolated points,  $c_{\mathrm s}``[Y]^2 = 2$. This uses the fact that any non-empty $Y \s \mathbb R$ which is dense-in-itself contains a set order-isomorphic to $\mathbb Q$, and hence also contains a set order-isomorphic to $\mathbb Z$.\footnote{Some details are provided in  Section~\ref{sectionnotation}.}

Since any set of reals homeomorphic to $\mathbb Q$ is dense-in-itself, it follows that the number of colours in Theorem~A and Corollary~B cannot be reduced to $1$, that is, we cannot obtain a monochromatic set.

Sierpi{\'n}ski's idea can also be applied to the rational numbers $\mathbb Q$. That is, starting from a well-ordering of $\mathbb Q$ instead, we obtain $c:[\mathbb Q]^2 \rightarrow 2$ such that for any set $Y \s \mathbb Q$ order-isomorphic to $\mathbb Z$, $c``[Y]^2 = 2$. It follows that there is no direct analogue of Ramsey's Theorem for $\mathbb Q$ as well. 

Galvin was then able to show in an unpublished result (from 1969 per \cite[pg.~71]{MR2628717}) that the Ramsey theorem holds for $\mathbb Q$ \emph{modulo this failure} (see discussion of \cite[Problem~18]{MR357122} and \cite[Chapter~4]{MR2628717} and for a proof see \cite[Section~6.3]{MR2603812}):
\begin{etheorem}[Galvin] Let $K$ be a positive natural number. For any colouring $c:[\mathbb Q]^2 \rightarrow K$, there is a set $Y \s \mathbb Q$ order-isomorphic to $\mathbb Q$ such that $|c``[Y]^2| \leq 2$.
\end{etheorem}
It follows that in Theorem~A and Corollary~B if we only desire an order-isomorphic copy of $\mathbb Q$, then one can simply colour any countably infinite subset of $\mathbb R$ which is dense as a linearly-ordered set. In order to see that the conclusion of Theorem~A and Corollary~B is stronger, recall the aforementioned fact that any set which is homeomorphic to $\mathbb Q$ contains a set order-isomorphic to $\mathbb Q$. 

Van Douwen asked the natural question whether Galvin's Theorem can be strengthened to obtain in fact a homeomorphic copy of $\mathbb Q$ on which at most $2$ colours are realised. This was answered negatively by Baumgartner in 1982.\footnote{For a positive result when the colourings are required to be continuous, see \cite[Theorem~4.1]{MR1297180}.}
\begin{etheorem}[Baumgartner, \cite{MR867644}] Let $X$ be a countable Hausdorff topological space. Then there is $c: [X]^2 \rightarrow \mathbb N$ such that for every $Y \s X$ homeomorphic to $\mathbb Q$, $c``[Y]^2 = \mathbb N$. 
\end{etheorem}
It follows that in Theorem~A the set of reals $X$ necessarily must be uncountable. 

Mentioned explicitly in Baumgartner's paper \cite{MR867644} from 1982 as ``the most interesting open problem in this area", and credited to the same letter, \cite{galvin}, from Galvin to Laver in 1970 by Raghavan and Todor{\v c}evi{\' c} in \cite{MR4190059} is the conjecture of Galvin which Corollary~B verifies. 
\begin{econjecture}[Galvin] For every positive natural number $K$ and $c:[\mathbb R]^2 \rightarrow K$ a colouring of pairs of reals, there is a set of reals $Y$ homeomorphic to $\mathbb Q$ such that $|c``[Y]^2| \leq 2$.
\end{econjecture}
This appears also in Weiss' survey on partition relations for topological spaces, \cite[Question~4.2]{MR1083599}. We will refer to the conjectured statement, that is, the statement of Corollary~B, as \emph{Galvin's Conjecture} in the rest of this introduction.

Related to this is the following question which Galvin also considered, but which appears already in the Erd{\H o}s-Hajnal problem list from 1967 as \cite[Problem~15]{MR280381}.
\begin{equestion}[Erd{\H o}s and Hajnal, \cite{MR280381}] Without assuming the continuum hypothesis, $2^{\aleph_0}= \aleph_1$, can one prove that there is a colouring $c:[\mathbb R]^2 \rightarrow 3$ of pairs of reals into $3$ colours, such that for every uncountable set of reals $Y$, $|c``[Y]^2| = 3$?
\end{equestion}
The assumption of the negation of the continuum hypothesis is justified by the Erd{\H o}s-Hajnal-Rado result (\cite[Theorem~17]{MR202613}) that assuming $|\mathbb R|= \aleph_1$, there is a colouring $c:[\aleph_1]^2\rightarrow \aleph_1$ such that on any uncountable $Y \s \aleph_1$, $c``[Y]^2 = \aleph_1$. Later in 1984, \cite{MR908147}, Todor{\v c}evi{\' c} using his theory of `walks on ordinals' (see the monograph \cite{MR2355670}) proved the spectacular result that the same is true just in $\zfc$, the most radical failure possible for the Ramsey Theorem at $\aleph_1$.\footnote{For various extensions of Todor{\v c}evi{\'c}'s result at $\aleph_1$ and at other cardinals, see the survey \cite{MR3271280}.}

These results show that the strengthening of Corollary~B where the set $Y$ (according to the notation there) is asked to be uncountable does not hold in any model satisfying the continuum hypothesis.

Returning to the Erd{\H o}s-Hajnal question, Shelah in 1985 proved that the answer to it is negative, by giving a forcing construction assuming the existence of suitable large cardinals. 
\begin{etheorem}[Shelah, \cite{Sh:276}] Assuming the existence of an $\omega_1$-Erd{\H o}s cardinal, the following is consistent: let $K$ be a positive natural number and let $c: [\mathbb R]^2 \rightarrow K$ be a colouring of pairs of reals into $K$-many colours; then for some uncountable $Y \s \mathbb R$, $|c``[Y]^2| \leq 2$.\footnote{Shelah later improved the large cardinal assumption in \cite[Theorem~1.21]{Sh:546} to a strongly Mahlo cardinal. See also \cite{Sh:288} for a similar result for all finite dimensions.}
\end{etheorem}
The exact definition of an $\omega_1$-Erd{\H o}s cardinal, or more generally, any other large cardinal we mention will not be required in the paper so we refer the reader to \cite{MR1994835} for details.
 
To connect the Erd{\H o}s-Hajnal question to Galvin's Conjecture, note that in Shelah's model, for every colouring $c: [\mathbb R]^2 \rightarrow K$ of pairs of reals into finitely-many colours, for some uncountable $Y \s \mathbb R$, $|c``[Y]^2| \leq 2$. Then we can shrink $Y$ to a countably infinite set $Z$ which is dense-in-itself, and hence $Z$ is homeomorphic to $\mathbb Q$\footnote{Sierpi{\'n}ski \cite{Sierpinski1920} proved that any countable non-empty metric space which is dense-in-itself is homeomorphic to $\mathbb Q$}. In other words, Shelah's forcing construction establishes that, assuming large cardinals, Galvin's Conjecture is consistent.

More recently in 2018 Raghavan and Todor{\v c}evi{\' c} made a breakthrough by verifying the conjecture axiomatically: assuming large cardinals, or hypotheses which have large cardinal strength.
\begin{etheorem}[Raghavan and Todor{\v c}evi{\' c}, \cite{MR4190059}]
Suppose that there is a Woodin cardinal or a strongly compact cardinal or a precipitous ideal on $\aleph_1$. Then Galvin's Conjecture holds.
\end{etheorem}
The advantage of the Raghavan-Todor{\v c}evi{\' c} result over the result of Shelah is that by the L{\'e}vy-Solovay phenomenon (see \cite[pg.~125]{MR1994835}) it follows that assuming there is a proper class of Woodin or strongly compact cardinals, not only does Galvin's Conjecture hold, but furthermore in any set forcing extension of the universe, Galvin's Conjecture continues to hold. In particular, a set forcing extension of an arbitrary model of set theory cannot produce a model where Galvin's Conjecture fails, which is strong evidence that Galvin's conjecture holds outright.\footnote{In contrast, for the Erd{\H o}s-Hajnal question, Baumgartner observed (see \cite[pg.~274]{MR357122}) that it is consistent that $|\mathbb R|$ is anything reasonable and there is $c:[\mathbb R]^2 \rightarrow \mathbb N$ such that on any uncountable $Y \s \mathbb R$, $c``[Y]^2 = \mathbb N$.}

Further evidence is given by Eisworth \cite{eisworth2023galvins} who improved the hypotheses to a Ramsey cardinal or a weakly precipitous ideal on $\aleph_1$ and in particular showed that Galvin's Conjecture may hold in a generic extension of G{\"o}del's constructible universe as well.

This brings us to our main result, Corollary~B, that Galvin's Conjecture is true in $\zfc$. This also answers the $2$-dimensional case of \cite[Open~Problem~8.2]{MR4253680} as well as \cite[Questions~1,2,3]{eisworth2023galvins}. Our inspiration is a technique of Shelah from \cite{Sh:881} which uses an idea from his work on PCF theory (PCF stands for `possible cofinalities', see the monograph \cite{Sh:g}). After we discuss the structure of the paper we will discuss our proof and compare it with those in \cite{MR4190059} and \cite{eisworth2023galvins}. 

As for the matter of whether Theorem~A and Corollary~B can be strengthened, it is easy to construct using the standard metric on $\mathbb R$ a colouring $c:[\mathbb R]^2 \rightarrow \mathbb N$ which takes infinitely-many values on $[Y]^2$ when $Y$ is not discrete.\footnote{There are some details in Lemma~\ref{lemma13}. This cannot be strengthened to realising all the colours on such a $Y$ by Corollary~B; on the other hand Galvin and Shelah \cite{Sh:23} constructed a $c:[\mathbb R]^2\rightarrow \mathbb N$ such that for any $Y \s \mathbb R$ with $|Y| = |\mathbb R|$, $c``[Y]^2 = \mathbb N$.} 

More interesting is the case of considering higher-dimensional colourings such as colouring triples of reals. 
Here Raghavan and Todor{\v c}evi{\' c} have proved the following, generalising Baumgartner's Theorem above which would correspond to $n=0$ below.
\begin{etheorem}[Raghavan and Todor{\v c}evi{\' c}, \cite{MR4579382}] For every positive natural number $n$, and a set of reals $X$ of size $\aleph_n$, there is $c:[X]^{n+2} \rightarrow \mathbb N$ such that for any $Y \s X$ which is homeomorphic to $\mathbb Q$, $c``[Y]^{n+2} = \mathbb N$.
\end{etheorem}
In the simplest case, it follows that we cannot improve Corollary~B to colouring triples of reals unless we also assume that $|\mathbb R| \geq \aleph_2$, that is, unless we assume the failure of the continuum hypothesis.

Before moving on to a discussion of the more general Theorem~D, we mention a corollary which improves on \cite[Corollary~8]{MR4190059} since our result is in $\zfc$. A motivation for this result comes from the discussion of the \emph{expansion problem} for a topological space (see \cite[\S~2]{MR4190059} for details) and the following results connects to Ramsey's original application of the Ramsey Theorem  in \cite{MR1576401} and a related result of Erd{\H o}s-Rado \cite{MR65615} (see the discussion of \cite[Corollary~8]{MR4190059}).
\begin{corc} Let $<$ denote the usual order on $\mathbb R$ and let $<^*$ be a well-ordering of $\mathbb R$. Let $X \s \mathbb R$ be an uncountable set, and let $E\s X^2$ be a binary relation (not necessarily symmetric). Then there is $Y \s X$ which is homeomorphic to $\mathbb Q$ such that $E \cap Y^2$ is one of the following restricted to $Y$: $\top, \perp, =, \neq, <, >, \leq, \geq, <^*, >^*, \leq^*, \geq^*, < \cap <^*, < \cap >^*, > \cap <^*, > \cap >^*, \leq \cap <^*, \leq \cap >^*, \geq \cap <^*, \geq \cap >^*$. 
\end{corc}
\begin{proof} The proof is similar to that of \cite[Theorem~1.7]{MR2603812} using Theorem~A instead of Ramsey's Theorem.
\end{proof}
Closely related to the expansion problem and more generally to the topic of this paper is the area of \emph{Ramsey degrees} which has recently seen a lot of progress, and we refer the interested reader to Dobrinen's ICM survey \cite{MR4680287} for this, connections with topological dynamics etc. Our Corollary~B or the forthcoming Theorem~D can both be refomulated using this language. Another survey where many related results are collected is \cite{MR4253680}. 

We end by mentioning that much of our knowledge of the history of this subject comes from Todor{\v c}evi{\'c's survey \cite{todorcevicbasis}.
\subsection{A Ramsey theorem for a class of topological spaces}\label{subsectiongeneral}
All topological spaces are assumed to be Hausdorff.
As we have mentioned earlier, Theorem~A is a consequence of a more general result that we prove. In this we are emulating Raghavan and Todor{\v c}evi{\' c}  \cite{MR4190059} who also proved Theorem~D below assuming additionally a class of Woodin cardinals or a strongly compact cardinal\footnote{This is the main theorem of their paper.}, and then use it to draw a corollary about Galvin's Conjecture, assuming the relevant large cardinal hypotheses.
Before we state Theorem~D, let us recall the definitions. The first is standard, the second 
is a property intermediate between being first countable and being second countable, and the last is due to Hajnal and Juh{\'a}sz, \cite{MR264585}.
\begin{definition} \label{defnclassofspaces}Let $(X, \tau)$ be a topological space. 
\begin{enumerate}
\item We say that $(X, \tau)$ is \emph{regular} if for every closed $C \s X$ and $x \in X \setminus C$, we can find pairwise disjoint open sets $U, V$ with $C \s U$ and $x \in V$. 
\item A base $\mathcal B \s \tau$ is said to be \emph{point-countable} if for every $x \in X$, $\{U \in \mathcal B \mid x\in U\}$ is countable.
\item We say that $(X, \tau)$ is \emph{left-separated} if there is $<^*$ a well-ordering of $X$ such that for every $x \in X$, $\{y\in X \mid y<^* x\}$ is a closed set.
\end{enumerate}
\end{definition}
 The property of left-separation is a notion of smallness among topological spaces, for example a countable Hausdorff space is always left-separated.
\begin{thmd} Let $(X, \tau)$ be a regular topological space which is not left-separated and has a point-countable basis. Let $K$ be a positive natural number and let $c:[X]^2 \rightarrow K$ a colouring of pairs of elements of $X$. Then there is a $Y\s X$ homeomorphic to $\mathbb Q$ such that $|c``[Y]^2| \leq 2$.\footnote{The reader unfamiliar with these topological notions will find in Section~\ref{sectionstart} the set-theoretic consequences of theirs which we use in the proof.}
\end{thmd}
\begin{proof}[Proof of Theorem~A]
In order to see that Theorem~A follows from Theorem~D, note that the reals have a countable basis, and that any uncountable set of reals has a countable subset which is dense in it. So, any uncountable set of reals satisfies the hypotheses of Theorem~D.
\end{proof}
The motivation for considering this specific class of topological spaces is laid out by  Raghavan and Todor{\v c}evi{\' c} in \cite[\S~3]{MR4190059}.
As they argue there, this is ``an essentially optimal class of topological spaces" for which a result such as Theorem~A holds. As for the matter of consistency strength, on \cite[pg.~22]{MR4190059} they say ``it would not be surprising if [Theorem~D] turns out to be equiconsistent with some large-cardinal axiom." An example is given from \cite{MR2304318} to justify this statement.
Eisworth \cite{eisworth2023galvins} improved the large cardinal bounds by obtaining Theorem~D from a class of Ramsey cardinals. As our result shows, $\zfc$ is enough.

Rather than merely repeating \cite[\S~3]{MR4190059} here to justify the class of topological spaces in Theorem~D, we refer the reader to it and focus here on the class of metrisable spaces where Theorem~D provides a provably optimal result. This too can be found in more detail in \cite[\S~2,\S~3]{MR4190059} so our treatment will be brief and we refer the reader to \cite{MR4190059} for a more comprehensive treatment.\footnote{According to \cite[pg.~9]{MR4190059} special stationary Aronszajn lines with the order topology are another example of spaces to which Theorem~D applies.}

It is clear that a metric space is regular Hausdorff; by the Nagata-Bing-Smirnov Metrisation Theorem (see \cite[\S~6]{MR3728284}), it also has a point-countable basis. So Theorem~D combined with the following lemma due to Fleissner yields a proof of Corollary~E below.
\begin{lemma}[Fleissner, \cite{MR825729}] Let $X$ be a metric space. The following are equivalent:
\begin{enumerate}
\item $X$ is left-separated;
\item $X$ is $\sigma$-discrete: there is a decomposition $X = \bigcup_{n< \omega} Y_n$ where each $Y_n$ is a discrete subspace with the relative topology.
\end{enumerate}
\end{lemma}
\begin{proof} For (i) implies (ii) see \cite[pg.~666]{MR825729} and for (ii) implies (i) see \cite[Theorem~2.2]{MR825729}.
\end{proof}
\begin{core} Let $X$ be a metric space which is not $\sigma$-discrete. Then for every positive natural number $K$ and colouring $c:[X]^2 \rightarrow K$, there is a $Y \s X$ homeomorphic to $\mathbb Q$ such that $|c``[Y]^2| \leq 2$.
\end{core}
The claim that Corollary~E is optimal in the class of metric spaces is justified by the following three lemmas. The first is trivial, showing that we cannot increase the number of colours.
\begin{lemma}\label{lemma13} Let $X$ be a metric space. Then there is $c:[X]^2 \rightarrow\mathbb N$ such that for every $Y \s X$ which is not discrete, $c``[Y]^2$ is infinite.
\end{lemma}
\begin{proof}[Proof sketch]
Let $d$ denote the metric on $X$. 
Let $c: [X]^2 \rightarrow \mathbb N$ be given by: for $x\neq y$ from $X$, 
$$c(x,y):= \begin{cases} \sup\{k \in \mathbb N \mid \frac{1}{k} \geq d(x,y)\} &\text{ if } d(x,y) \leq 1; \\
0&\text{ otherwise} .
\end{cases}$$
Then it is easy to see that $c$ is as required.
\end{proof}
The second, showing that an analogue of Sierpi{\'n}ski's colouring holds for metric spaces, we learnt from \cite[Theorem~12]{MR4190059}.
\begin{lemma} Let $X$ be a metric space. Then there is $c:[X]^2 \rightarrow\{0,1\}$ such that for every $Y \s X$ which is dense-in-itself, $c``[Y]^2 = \{0,1\}$.
\end{lemma}
\begin{proof}[Proof sketch] 
Every metric space can be embedded into a space of the form $[0,1]^\kappa$ for some cardinal $\kappa$. Considering $X$ to be a subset of $[0,1]^\kappa$ for some cardinal $\kappa$, recall the \emph{lexicographic order} $<_{\mathrm{lex}}$ on $[0,1]^\kappa$ is defined via: for distinct $f, g \in [0,1]^\kappa$,
$$f<_{\mathrm{lex}} g \iff f(\alpha) < g(\alpha) \text{ where }\alpha:=\min\{\beta< \kappa \mid f(\beta) \neq g(\beta)\}.$$
Then we can imitate Sierpi{\' n}ski's colouring by defining a function $c:[X]^2 \rightarrow 2$ using $<_{\mathrm{lex}}$ and $<^*$ where $<^*$ is some well-ordering of $X$. See \cite[Lemma~11]{MR4190059} and \cite[Theorem~12]{MR4190059} for details. 
\end{proof}
The last is the following unpublished result of Todor{\v c}evi{\' c} and Weiss which provides a negative result for any $\sigma$-discrete metric space.
\begin{lemma}[Todor{\v c}evi{\'c}-Weiss, \cite{todorcevicweiss}] Let $X$ be a $\sigma$-discrete metric space. Then there is $c: [X]^2 \rightarrow \mathbb N$ such that for every $Y \s X$ homeomorphic to $\mathbb Q$, $c``[Y]^2 = \mathbb N$.
\end{lemma}

\subsection{Structure of the paper}
We now describe the layout of the paper. 

In Section~\ref{sectionnotation} we collect in one place general notation and other preliminaries. 

In Section~\ref{sectionideals} we recall standard notation, terminology, and facts about ideals on countable sets. In Subsection~\ref{subsectionJdf} we generalise to this context an operator on ideals  introduced by Shelah in his work on PCF theory. We will use this in the proof of the main lemma.

In Section~\ref{sectionstart} we start the proof of Theorem~D. In particular we fix here various objects we shall use throughout the proof. We also explain here the role of the topological hypotheses of Theorem~D. Beyond this point the proof of Theorem~D proceeds in three steps. 

In Section~\ref{sectionRT} we recall the idea of Raghavan and Todor{\v c}evi{\'c} which provides us with a sufficient condition to construct the homeomorphic copy of the rationals as required. The main result here is Theorem~\ref{RaghavanTodorcevicconstruction}. This is step one. The main lemma, which is step three, is to show that this sufficient condition is met. 

In Section~\ref{sectionstartingpoint} we perform step two, which is to find the appropriate input for our main lemma. The main result here is Lemma~\ref{firstmainlemma}.

In Section~\ref{sectionmainlemma} we prove the main lemma, Lemma~\ref{Shelahidea} and conclude the proof of Theorem~D.
\subsection{Discussion of the proof}
As we have mentioned, our main result is the same as that of \cite{MR4190059} with the difference being that we work only in $\zfc$. We now compare our result with \cite{MR4190059}, the reader can look there for undefined notions. 

There are two places where large cardinals are required in \cite{MR4190059}. One is simply for the definition of Woodin's stationary tower forcing (see \cite{MR2723878} or \cite{MR2069032}), this needs a class of strongly inaccessible cardinals. The second is for the precipitousness of the stationary tower, this needs considerably more large cardinals: a class of Woodin cardinals or a strongly compact cardinal. Eliminating large cardinals from the first step is somewhat straightforward, the reader may compare the results of Section~\ref{sectionstart} with \cite[Lemma~26]{MR4190059}.

The key argument is in the proof of the main lemma, Lemma~\ref{Shelahidea}, in Section~\ref{sectionmainlemma}. This replaces the use of games related to precipitousness in \cite{MR4190059} and \cite{eisworth2023galvins} by an argument using ranks inspired by \cite{Sh:881}. \footnote{The argument can be translated into the language of precipitousness-type games however we do not do so over here.} 
In particular, this step achieves the more substantial reduction of large cardinals to a $\zfc$ argument. 

The necessary machinery for ranks is developed in Subsection~\ref{subsectionJdf}.
To make this machinery applicable we need to generalise the basic idea from \cite{MR4190059} by working not just with the fixed ideal of non-stationary sets, but allowing the ideal to vary. The idea of varying the ideals is present in \cite[\S~3.1]{eisworth2023galvins} as well, though our treatment is rather different.
In this greater generality we then repeat some arguments from \cite{MR4190059}. This is done in Section~\ref{sectionstartingpoint}.

The last piece is to restrict ourselves to the class of fine normal ideals, see Lemma~\ref{onestep}.

It should be clear by now that our paper heavily builds on the ideas of \cite{MR4190059}. However, there are enough minor differences that we have taken the decision to make the paper self-contained.
\section{Notation and preliminaries}\label{sectionnotation}
For a set $X$ we denote its powerset by $\p(X)$ and its cardinality by $|X|$ and the set of its ordered pairs by $X^2$ and the set of its two-element subsets by $[X]^2$. Note that this will also apply to natural numbers, so for $K$ a natural number, by $K^2$ we will designate the set of all ordered pairs $(i,j)$ for $i, j< K$ not necessarily distinct. For a function $f: X \rightarrow Y$ and $Z \s X$, $f``[Z]:=\{f(x) \mid x \in Z\}$.

By $\pom(X)$ we denote the collection of countable subsets of $X$,
 $$\pom(X):=\{M \s X \mid |M| < \aleph_1\}.$$
Ideals on sets of the form $\pom(X)$ play a central role in this paper. Notation and some important facts on them are collected in Section~\ref{sectionideals}. 

We will consider colourings $c: [X]^2 \rightarrow K$ for $X$ some set and $K$ a natural number. We will use the convention that if $x, y \in X$ and $c(x,y)$ has been defined, has some specific value etc., then there is an implicit assumption that $x \neq y$. On the other hand if we only know that $c(x,y) \neq i$ for some $i< K$ then this may also be because in fact $x=y$. 

We will talk about trees of finite sequences at some points, so let us refresh the reader's memory and introduce some relevant terminology we will use as well. For $Y$ some set, ${}^{< \omega}Y$ denotes the set of finite sequences of elements of $Y$.  For $\sigma, \rho$ finite sequences of elements of some set $Y$,  $\sigma \sqsubset \rho$ shall denote that $\sigma$ is a proper initial segment of $\rho$. The empty sequence we denote by $\langle\rangle$. If $\sigma, \rho$ are two elements of ${}^{< \omega}Y$, then $\sigma {}^\smallfrown \rho$ denotes the concatenation of $\sigma$ and $\rho$. 

We say that $\mathcal T \s {}^{< \omega}Y$ is a \emph{tree} if it is closed under taking initial segments. For $Y$ some set, and $\mathcal T \s {}^{< \omega}Y$ a tree, it is \emph{well-founded} if for every $b \in {}^\omega Y$, for some $k< \omega$, $b \restriction k \notin \mathcal T$. In this case we can define a ordinal-valued rank function. The details will be provided. 

For subtrees of ${}^{< \omega}\omega$, we use $<_{\mathrm{lex}}$ to refer to the lexicographic order, so that for $\sigma, \rho \in{}^{< \omega}\omega$ distinct and such that $\sigma \not \sqsubset  \rho$ and $\rho \not \sqsubset  \sigma$, $\sigma <_{\mathrm{lex}} \rho$ if for the least $i$ such that $\sigma(i) \neq \rho(i))$, we have that $\sigma(i) < \rho(i)$.

For a cardinal $\Omega$, $H(\Omega)$ denotes the collection of all sets of hereditary cardinality less than $\Omega$. In the context of taking elementary submodels by $H(\Omega)$ we shall mean the structure $(H(\Omega), \in, \triangleleft)$ where $\triangleleft$ will denote some fixed well-ordering of $H(\Omega)$. We will denote that $M$ is an elementary submodel of this structure by $M \prec H(\Omega)$.

Lastly we discuss some topological preliminaries. All spaces are assumed to be Hausdorff. The main class of spaces we study has been defined in Definition~\ref{defnclassofspaces}.

For $(X, \tau)$ be a topological space and $M \s X$, $\overline{M}$ will denote the closure of $M$. If  $\mathcal B \s \tau$ is a base for the topology, then for $Y \s X$, we let 
$$\mathcal B_Y:=\{U \in \mathcal B \mid U \cap Y \neq \emptyset\}.$$

Given a set $X$ and a point $x\notin X$, we will say that $X$ is \emph{dense around $x$} if every open neighbourhood of $x$ contains a point of $X$. We will say that $X$ is \emph{dense-in-itself} if for every $x \in X$, $X \setminus \{x\}$ is dense around $x$, in other words, $x$ is not an isolated point of $X$. 

An old theorem of Sierpi{\'n}ski \cite{Sierpinski1920} which we have mentioned before is that any countable non-empty metric space which is dense-in-itself is homeomorphic to $\mathbb Q$. 
Recall also the Urysohn Metrisation Theorem which states that any second countable regular Hausdorff space is metrisable. 

Finally we justify some statements from the introduction. Let $X \s \mathbb R$ be uncountable subset of the reals, then there is $Y \s X$ also uncountable such that $Y$ is dense in itself. Indeed, let $\mathcal U$ denote the set of all rational intervals $U$ such that $U \cap X$ is countable, and then let $Y:= X \setminus (\bigcup \mathcal U)$. Given such an $Y$, there is $Z\s Y$ countably infinite such that $Z$ is also dense-in-itself. 

We also mentioned that any $X \s \mathbb R$ which is homeomorphic to $\mathbb Q$ contains a set order-isomorphic to $\mathbb Q$, so let us quickly check the details. First we define an equivalence relation $\sim$ on $X$ via: $x \sim y$ if $(x, y) \cap X = \emptyset$. Then each equivalence class has size at most $2$. So let $Y \s X$ be such that $Y$ contains exactly one element of each equivalence class of $\sim$. So $Y$ must be countably infinite too. Then for any $x < y$ from $Y$, $(x, y) \cap Y \neq \emptyset$. In other words, $Y$ is dense as a linear order. Now discarding the end points of $Y$, of which there are at most two, we obtain a set $Z$ which is a countable dense linear order with no endpoints, so by Cantor's Theorem $Z$ is order-isomorphic to $\mathbb Q$, and hence every countable linear order can be order-embedded into $Z$. 
\section{Ideals of countable sets}\label{sectionideals}
Recall that for $H$ a non-empty set, $\pom(H)$ denotes the set $\{M \s H \mid |M| < \aleph_1\}$. We start with standard notation, terminology, and facts about $\pom(H)$ which can be found, for example, in \cite{MR1994835}. We finish with Subsection~\ref{subsectionJdf} which generalises some ideas of Shelah \cite{Sh:71} to this context.

The following can be found in \cite[\S~25]{MR1994835}.
\begin{definition} Let $H$ be a non-empty set. 
\begin{enumerate}
\item A set $ C \s \pom(H)$ is \emph{club in $\pom(H)$} if the following holds:
\begin{itemize} 
\item for every $N \in \pom(H)$ there is an $M \in C$ with $ N \s M$. 
\item for every $0< \xi< \omega_1$ and a $\s$-increasing sequence $\langle M_\nu \mid \nu<\xi\rangle$ of elements of $C$, $\bigcup_{\nu< \xi}M_\xi \in C$.
\end{itemize}
\item A set $S \s \pom(H)$ is \emph{stationary in $\pom(H)$} if for every $C\s \pom(H)$ club in $\pom(H)$, $S \cap C\neq\emptyset$.
\end{enumerate}
\end{definition}
\begin{fact} \label{lemma32}
Let $X \s H$ be non-empty sets and let $S \s \pom(X)$ and $T \s \pom(H)$.
\begin{enumerate}
\item If $S$ is stationary in $\pom(X)$ then $ S^{\uparrow H}$ is stationary in $\pom(H)$ where 
 $$S^{\uparrow H}:=\{M \in \pom(H)\mid M \cap X \in S\}.$$
\item If $T$ is stationary in $\pom(H)$ then $ T^{\downarrow X}$ is stationary in $\pom(X)$ where 
 $$ T^{\downarrow X}:= \{M \cap X \mid M \in T\}.$$
\end{enumerate}
\end{fact}
The following can be found in \cite[pg.~301 and \S~25]{MR1994835}.
\begin{definition} Let $H$ a non-empty set. Let $\mathcal I \s \p(\pom(H))$. 
\begin{enumerate}
\item $\mathcal I$ is a \emph{ideal on $\pom(H)$} if 
\begin{itemize}
\item for every $A, B \in \mathcal I$, $A \cup B \in \mathcal I$,
\item for every $A \in \mathcal I$ and $B \s A$, $B \in \mathcal I$, and
\item $\pom(H)\notin \mathcal I$.
\end{itemize}
\item An ideal $\mathcal I$ is \emph{$\aleph_1$-complete} if given a sequence $\langle A_n \mid n< \omega \rangle$ of elements of $\mathcal I$, $\bigcup_{n< \omega}A_n\in \mathcal I$.
\item An $\aleph_1$-complete ideal $\mathcal I$ is a \emph{fine ideal on $\pom(H)$} if for every $x \in H$, $\{M \in\pom(H)\mid x \not\in M\} \in \mathcal I$.
\item For a sequence $\langle A_x \mid x \in H\rangle$ of elements of $\pom(H)$, the \emph{diagonal union} of  $\langle A_x \mid x \in H\rangle$, denoted $\diagonalunion_{\!x \in H} A_x$, is the set $\{M \in \pom(H) \mid M \in \bigcup_{x \in M} A_x\}$.
\item For a sequence $\langle A_x \mid x \in H\rangle$ of elements of $\pom(H)$, the \emph{diagonal intersection} of  $\langle A_x \mid x \in H\rangle$, denoted $\diagonalintersection_{x \in H} A_x$, is the set $\{M \in\pom(H)\mid M \in \bigcap_{x \in M} A_x\}$.
\item $\mathcal I$ is \emph{normal} if for every sequence $\langle A_x \mid x \in H\rangle$ of elements of $\mathcal I$, $\diagonalunion_{\!x \in H} A_x \in \mathcal I$. 
\item For an ideal $\mathcal I$ the dual filter, denoted $\mathcal I^*$, is the set
$$\mathcal I^*:= \{A \s\pom(H)\mid \pom(H) \setminus A \in\mathcal I\}.$$
\item For an ideal $\mathcal I$ the collection of all $\mathcal I$-positive sets, denoted $\mathcal I^+$, is the set
$$\mathcal I^+:= \{A \s\pom(H) \mid A \notin \mathcal I\}.$$
\end{enumerate}
\end{definition}
So by default we assume all ideals are proper ideals and that fine ideals are $\aleph_1$-complete. An example of a fine normal ideal is the collection of non-stationary subsets of $\pom(H)$ for $H$ an uncountable set.

Given $H'\s H$ and a sequence $\langle A_x \mid x\in H'\rangle$ of elements of $\pom(H)$, we will let $\diagonalunion_{\!x \in H'}A_x$ denote the set $\diagonalunion_{\!x \in H}B_x$ where for $x \in H'$ we take $B_x:= A_x$ and for $x \in H\setminus H'$ we take $B_x:= \emptyset$. We can similarly take diagonal intersections of a sequence indexed by a subset of $H$ by adding $\pom(H)$ at the missing indices. Note that given a sequence $\langle A_x \mid x\in H\rangle$ if $H= H_0\cup H_1$, then 
$$(\diagonalunion_{\!x \in H_0}A_x) \cup (\diagonalunion_{\!x \in H_1}A_x) = \diagonalunion_{\!x\in H}A_x.$$

For $A \s \pom(H)$, we will call a function $\Phi:A \rightarrow H$ a \emph{choice function} if for every $M \in A$, $\Phi(M) \in M$. 

The following are some standard facts which serve also to introduce some notation we shall use throughout the paper.
\begin{fact}\label{lemmaidealfacts} Let $H$ be a non-empty set and $\mathcal I$ an ideal on $\pom(H)$.
\begin{enumerate}
\item If $\mathcal I$ is a fine normal ideal, then $\mathcal I$ extends the ideal of non-stationary subsets of $\pom(H)$. 
\item $\mathcal I$ is normal if and only if $\mathcal I^*$ is closed under diagonal intersections.
\item Let $A \in \mathcal I^+$ and let $\Phi:A \rightarrow H$ be a choice function. Then there is $B \s A$, $B \in \mathcal I^+$, and an $x \in H$ such that $\Phi``[B] = \{x\}$.
\item For $A \in \mathcal I^+$, there is a smallest ideal on $\pom(H)$ containing $\mathcal I \cup\{\pom(H) \setminus A\}$ which we denote by $\mathcal I+A$, and it is described explicitly by the following: for $B \s\pom(H)$, 
$$B \in \mathcal I +A \iff B \cap A \in \mathcal I.$$
\item For $A \in \mathcal I^+$, the dual filter to $\mathcal I+A$, $(\mathcal I+A)^*$, is described explicitly by the following: for $B \s \pom(H)$, 
$$B \in (\mathcal I +A)^* \iff B \cup (\pom(H) \setminus A) \in \mathcal I^*.$$
\item If $\mathcal I$ is a fine normal ideal and $A \in \mathcal I^+$, then $\mathcal I+A$ is a fine normal ideal as well.
\item Suppose that $\mathcal I$ is a fine ideal. Let $\langle A_x \mid x \in H\rangle$ be a sequence of elements of $\pom(H)$ such that $A \in \mathcal I^+$, where $A:=\diagonalintersection_{x \in H} A_x$. Then, for every $x \in H$, $A_x \in (\mathcal I+A)^*$.
\end{enumerate}
\end{fact}
For $H$ a non-empty set and $\mathcal I$ an ideal on $\pom(H)$ and $A \in \mathcal I^+$ we will use $\exists^{\mathcal I} M\in A\,\phi(M)$ to abbreviate that 
$$\{M\in A \mid \phi(M)\} \in (\mathcal I+A)^+,$$ 
and 
$\forall^{\mathcal I} M\in A \,\phi(M)$ to abbreviate that 
$$\{M \in A \mid \phi(M)\}\in (\mathcal I+A)^*.$$

\subsection{An operator on ideals following Shelah}\label{subsectionJdf}
Our purpose in this subsection is to generalise to ideals on $\pom(H)$ some ideas implicit in \cite{Sh:71}, one of Shelah's first papers on PCF theory. The definitions of the ranks below are variants of the ranks first used by Galvin and Hajnal \cite{MR376359}. This was continued by Shelah in \cite{Sh:71} and more. Definition~\ref{defnJdf} can be compared with  \cite[Remark~20(B)]{Sh:71} and Lemma~\ref{JDFfacts} with \cite[Lemma~19]{Sh:71}. More explicitly, what follows is an analogue of the $\mathbf J[f, D]$ operator of Shelah, see \cite[\S~5]{Sh:589}.

For the rest of this section, fix $H$ a non-empty set. For $\mathcal I$ an ideal on $\pom(H)$ and functions $f:\p_{\!\omega_1}(H)\rightarrow \ord$ and $g:\p_{\!\omega_1}(H)\rightarrow \ord$, $f<_{\mathcal I} g$ denotes that 
$$\{M \in \pom(H) \mid \neg( f(M) < g(M))\} \in \mathcal I.$$ 
We similarly define the relations $\leq_\mathcal I$ and $=_\mathcal I$.

When $\mathcal I$ is $\aleph_1$-complete we have that the relation $<_{\mathcal I}$ is well-founded on ${}^{\pom(H)} \ord$.
Hence we can define an ordinal-valued rank function $\rk_{\mathcal I}$ on functions $f:\pom(H)\rightarrow \ord$ recursively as follows: $\rk_{\mathcal I}(f)$ is the least ordinal $\alpha$ such that for every $g :\pom(H)\rightarrow \ord$, 
$$g<_\mathcal I f \implies \rk_{\mathcal I}(g) < \alpha.$$

It should be clear that for $\aleph_1$-complete ideals $\mathcal I \s \mathcal J$ on $\pom(H)$ and $f: \pom(H) \rightarrow \ord$, 
$$\rk_{\mathcal I}(f) \leq \rk_{\mathcal J}(f).$$
In particular, for $A \s B$ both elements of $\mathcal I^+$ for $\mathcal I$ an $\aleph_1$-complete ideal on $\pom(H)$, as $\mathcal I+B \s \mathcal I+A$, 
$$\rk_{\mathcal I+B}(f) \leq \rk_{\mathcal I+A}(f).$$

For the rest of this section, fix $\mathcal I$ an $\aleph_1$-complete ideal on $\pom(H)$. If $Z \in \mathcal I^*$ and $f, g: \pom(H)\rightarrow \ord$ are such that $f \restriction Z = g \restriction Z$, then we have that $\rk_{\mathcal I}(f) = \rk_{\mathcal I}(g)$, and for this reason we can talk about $\rk_{\mathcal I}(h)$ for any ordinal-valued function $h$ so long as $\dom(h) \in \mathcal I^*$. 

For completeness we point out the following simple fact.
\begin{fact}\label{lemma31}Let  $f: \pom(H)\rightarrow \ord$ and let $\alpha:=\rk_{\mathcal I}(f)$. 

Then for every $\beta< \alpha$, there is a $g: \pom(H) \rightarrow \ord$ with $g<_{\mathcal I} f$ and $\rk_{\mathcal I}(g) = \beta$.
\end{fact}

\begin{definition}\label{defnJdf}
Let $f : \pom(H)\rightarrow\ord$. 
Then
$$\mathbf I[f, \mathcal I]:= \mathcal I \cup \{A \in \mathcal I^+ \mid \rk_{\mathcal I+A}(f) > \rk_\mathcal I(f)\}.$$
\end{definition}
\begin{lemma}\label{JDFfacts} Let $f: \pom(H)\rightarrow \ord$ and let $\mathcal I$ be a fine ideal on $\pom(H)$.
\begin{enumerate}
    \item $\mathbf I[f, \mathcal I]$ is a fine ideal on $\pom(H)$ extending $\mathcal I$.
    \item If $\mathcal I$ is a normal ideal then $\mathbf I[f, \mathcal I]$ is a normal ideal.
\end{enumerate}  
\end{lemma}
\begin{proof} 
It is clear that $\mathbf I[f, \mathcal I] \supseteq \mathcal I$. Since $\pom(H) \in \mathcal I^*$, $\mathcal I+ \pom(H) = \mathcal I$, and so it is clear that $\pom(H) \notin \mathbf I[f, \mathcal I]$. 
For $A \s B$ both in $\mathcal I^+$, we have that $\rk_{\mathcal I+B}(f) \leq \rk_{\mathcal I+A}(f)$, so it follows that $\mathbf I[f, \mathcal I]$ is closed under taking subsets.
For $A \in \mathcal I^+$ and $B \in \mathcal I$, $\mathcal I+(A\cup B)= \mathcal I+A$. This makes it clear that for $A \in \mathbf I[f, \mathcal I]$ and $B \in \mathcal I$, $A \cup B \in \mathbf I[f, \mathcal I]$.

Given a sequence $\langle A_n \mid n< \omega\rangle$ of elements of $\mathbf I[f, \mathcal I]$, we can find a pairwise disjoint sequence $\langle B_n \mid n< \omega\rangle$ such that for every $n< \omega$, $B_n \s A_n$, and $\bigcup_{n< \omega}A_n = \bigcup_{n< \omega}B_n$. It may be that some of the $B_n$ are elements of $\mathcal I$ and some others not, but combined with the previous observations and the fact that $\mathcal I$ is $\aleph_1$-complete the following claim proves that $\mathbf I[f, \mathcal I]$ is an ideal and $\aleph_1$-complete, hence a fine ideal. 
\begin{claim} Let $\langle B_n\mid n< \omega\rangle$ be a sequence of pairwise disjoint elements of $\mathcal I^+$, and let $B:= \bigcup_{n< \omega}B_n$.  Let $\langle f_n\mid n< \omega \rangle$ be a sequence of functions in ${}^{\pom(H)} \ord$ such that for every $n< \omega$, $f_n \leq_{\mathcal I+B_n} f$.

Then 
\begin{enumerate}
\item  for every ordinal $\zeta$, if $\zeta \leq  \min\{\rk_{\mathcal I+B_n}(f_n) \mid n< \omega\}$, then $\zeta \leq \rk_{\mathcal I+B}(f)$;
\item $\rk_{\mathcal I+B}(f) \geq \min\{\rk_{\mathcal I+B_n}(f_n) \mid n< \omega\}$.
\end{enumerate}
\end{claim}
\begin{why}
It is clear that (i) implies (ii). We will prove (i) by induction on $\zeta$.

Since the rank is always at least $0$, it is clear for $\zeta= 0$, and it is also clear for $\zeta$ limit. Now, suppose that $\zeta= \eta+1$. Then by Fact~\ref{lemma31} there are $\langle g_n \mid n< \omega \rangle$ such that for every $n< \omega$, $g_n<_{\mathcal I+B_n} f_n$ and $\rk_{\mathcal I+B_n}(g_n) \geq \eta$. Let $g: \pom(H) \rightarrow \ord$ be a function such that for $n< \omega$, $g\restriction B_n= g_n \restriction B_n$. Applying the induction hypothesis to $\langle g_n \mid n< \omega\rangle$ and $g$ and $\eta$, we see that $\rk_{\mathcal I+B}(g) \geq \eta$. So, if we show that $g<_{\mathcal I+B} f$ then we will be done. Suppose that this is not so, so that 
$$\{M \in B \mid g(M) \geq f(M)\} \in \mathcal I^+.$$
As $\mathcal I$ is $\aleph_1$-complete, it follows that there is some $n< \omega$ such that 
$$\{M \in B_n \mid g(M) \geq f(M)\} \in \mathcal I^+.$$
This contradicts that 
$$g =_{\mathcal I +B_n} g_n<_{\mathcal I+B_n} f_n \leq_{\mathcal I +B_n} f,$$
and we finish.
\end{why}
We now wish to show that if $\mathcal I$ is normal, then so is $\mathbf I[f, \mathcal I]$. So, suppose that $\mathcal I$ is normal, and let $\langle A_x \mid x\in H\rangle$ be a sequence of elements of $\mathbf I[f, \mathcal I]$ and let $A:= \diagonalunion_{\!x \in H} A_x$. We wish to show that $A \in \mathbf I[f, \mathcal I]$ too. 

Fix $<^*$ a well-order of $H$. Now, by recursion along $<^*$, for $x \in H$ define a set $B_x \s A_x$ as follows: for $M \in \pom(H)$, $M \in B_x$ if and only if $x$ is the $<^*$-least element such that
\begin{itemize}
\item $M \in A_x$, and
\item $x \in M$.
\end{itemize}
Note that for $M \in \pom(H)$, $M \in \bigcup_{x \in M}A_x$ if and only if $M \in \bigcup_{x \in M}B_x$. In other words, 
$$\diagonalunion_{\!x \in H}B_x = \diagonalunion_{\!x \in H}A_x = A,$$
and for every $M \in \pom(H)$, $\{x \in M \mid M \in B_x\}$ has cardinality at most $1$.

Furthermore, for every $x \in H$, $B_x \s A_x$, and as $\mathbf I[f, \mathcal I]$ is closed under taking subsets we have that for every $x \in H$, $B_x \in \mathbf I[f, \mathcal I]$ too. However, some of the $B_x$ may be in $\mathcal I$ and some others not. That motivates the following claim which combined with previous observations proves that $\mathbf I[f, \mathcal I]$ is normal.

\begin{claim} Let $H' \s H$ and let $\langle B_x \mid x \in H'\rangle$ be a sequence of elements of $\mathcal I^+$ such that for every $M \in \pom(H)$, $\{x \in M\cap H' \mid M \in B_x\}$ has cardinality at most $1$. Let $B:=\diagonalunion_{\!x \in H'}B_x$. Let $\langle f_x \mid x \in H'\rangle$ be a sequence of functions in ${}^{\pom(H)} \ord$ such that for every $x \in H'$, $f_x \leq_{\mathcal I+B_x} f$.

Then 
\begin{enumerate}
\item  for every ordinal $\zeta$, if $\zeta \leq  \min\{\rk_{\mathcal I+B_x}(f_x) \mid x\in H'\}$, then $\zeta \leq \rk_{\mathcal I +B}(f)$;
\item $\rk_{\mathcal I+B}(f) \geq \min\{\rk_{\mathcal I+B_x}(f_x) \mid x\in H'\}$.
\end{enumerate}
\end{claim}
\begin{why} It is clear that (i) implies (ii). We will prove (i) by induction on $\zeta$. 

 Since the rank is always at least $0$, it is clear for $\zeta= 0$, and it is also clear for $\zeta$ limit. Now, suppose that $\zeta= \eta+1$. Then by Fact~\ref{lemma31} there are $\langle g_x \mid x \in H' \rangle$ such that for every $x\in H'$, $g_x<_{\mathcal I+B_x} f_x$ and $\rk_{\mathcal I +B_x}(g_x) \geq \eta$. 

Let $g:\pom(H)\rightarrow \ord$ be a function such that for every $M \in \pom(H)$, if there is a (necessarily unique) $x \in M \cap H'$ such that $M \in B_x$, then $g(M) = g_x(M)$. We first observe that for $x \in H'$, $g =_{\mathcal I+ B_x} g_x$. Indeed, for $x \in H'$, as $B_x \in \mathcal I^+$ and $\mathcal I$ is a fine ideal, $\{M \in B_x \mid x \notin M\} \in \mathcal I+B_x$.

Now, applying the induction hypothesis to $\langle g_x \mid x \in H'\rangle$ and $g$ and $\eta$, we see that $\rk_{\mathcal I+B}(g) \geq \eta$. So, if we show that $g<_{\mathcal I+B} f$ then we will be done. 
So, let 
$$B':= \{M \in B \mid g(M) \geq f(M)\},$$
and suppose towards a contradiction that $B' \in \mathcal I^+$. So, let $\Phi: B' \rightarrow H$ be the choice function defined via: for $M \in B'$, $\Phi(M)$ is the unique $x \in M \cap H'$ such that $M\in B_{x}$. Then $\Phi$ takes a constant value (which we will denote by $x$, which is an element of $H'$) on a set in $\mathcal I^+$. So, the set $B' \cap B_{x} \in \mathcal I^+$, which means that $g \not<_{\mathcal I +B_x} f$

However, we also have that 
$$g =_{\mathcal I +B_x} g_x<_{\mathcal I+B_x} f_x \leq_{\mathcal I+B_x}f,$$
which means that we have reached a contradiction.
\end{why}
With that the proof is completed.
\end{proof}

\section{Start of  the proof}\label{sectionstart}
We start the proof of Theorem~D in this section and mainly fix a number of objects which we shall need in our proof. We start by explaining the role of the topological hypotheses. The objects are different, but the ideas are imported from Raghavan and Todor{\v c}evi{\'c}, \cite{MR4190059}, to this context. 

The role of left-separation is explained by the following result of Fleissner.
\begin{theorem}[Fleissner, \cite{MR825729}] \label{fleissnertheorem}
Let $(X, \tau)$ be a regular space with point-countable basis. Then $(X, \tau)$ is not left-separated if and only if $\{M \in \pom(X)  \mid\overline{M}\setminus M \neq \emptyset\}$ is stationary in $\pom(X)$.\footnote{Note that the result is true even for merely $T_1$ spaces instead of regular Hasudorff, however our extra hypotheses are for the purpose of finding homeomorphic copies of $\mathbb Q$. See the use of Urysohn's Metrisation Theorem in Theorem~\ref{RaghavanTodorcevicconstruction}.} 
\end{theorem}
We will also need the following consequence of a point-countable basis of a topological space.
\begin{lemma} \label{lemma42}Let $(X, \tau)$ be a topological space with $\mathcal B \s \tau$ a point-countable basis. Let $\Omega$ be an uncountable regular cardinal such that $X, \tau, \mathcal B \in H(\Omega)$. Let $M \prec H(\Omega)$ be a countable elementary submodel with $X, \tau, \mathcal B \in M$.

If $x \in \overline{M \cap X}$, then $\mathcal B_{\{x\}} \s M$.
\end{lemma}
\begin{proof} Let $U \in \mathcal B$ be such that $x \in U$. As $x \in  \overline{M \cap X}$, there is a $y \in M \cap X$ such that $y \in U$. Now as $\mathcal B_{\{y\}}$ is countable, $\mathcal B_{\{y\}} \s M$, so that $U \in M$. 
\end{proof}

We now start the proof of Theorem~D. First we fix some objects which we will use throughout the paper. 
\begin{itemize}
\item Let $(X, \tau)$ be a regular topological space which is not left-separated and let $\mathcal B \s \tau$ be a point-countable basis for it. 
\item Let $\Omega$ be an uncountable regular cardinal such that $X, \tau, \mathcal B \in H(\Omega)$.
\item Fix $\langle U_{x,k}\mid x \in X, k< \omega\rangle$ such that for every $x \in X$, $\langle U_{x, k} \mid k< \omega\rangle$ enumerates $\mathcal B_{\{x\}}$.
\end{itemize}
By Theorem~\ref{fleissnertheorem} there is $\Lambda_0 \s \pom(X)$ a stationary set such that for every $M\in \Lambda_0$, $\overline{M} \setminus M \neq \emptyset$.

The set of countable elementary $M \prec H(\Omega)$ containing $X, \tau, \mathcal B$ is club in $\pom(H(\Omega))$. By Fact~\ref{lemma32}(i), the set 
$$\{M \in \pom(H(\Omega)) \mid X, \tau, \mathcal B \in M,\, M \cap X \in \Lambda_0\}$$
is stationary in $\pom(H(\Omega))$. So fix $\Lambda_1 \s \pom(H(\Omega))$ a stationary set such that $M \in \Lambda_1$ implies that 
\begin{itemize}
\item $M \prec  H(\Omega)$, $|M|=\aleph_0$, and $X, \tau, \mathcal B\in M$, and 
\item $M \cap X \in \Lambda_0$ so that $\overline{M \cap X} \setminus M\neq \emptyset$.
\end{itemize}
Fix also a function $F:\Lambda_1 \rightarrow X$ such that for every $M \in \Lambda_1$, $F(M) \in \overline{M \cap X} \setminus M$. 

\begin{definition}\label{defC1}
The set $\mathcal C_1$ consists of pairs $(A, \mathcal I)$ where 
\begin{itemize}
\item $\mathcal I$ is a fine normal ideal on $\pom(H(\Omega))$, and 
\item $A \s \Lambda_1$, and $A \in \mathcal I^+$.
\end{itemize}
Given $(A, \mathcal I)$ and $(B, \mathcal J)$ from $\mathcal C_1$, we let $(B, \mathcal J) \leq_1 (A, \mathcal I)$ denote that 
\begin{itemize}
\item $B \s A$, and
\item $\mathcal J \supseteq \mathcal I$.
\end{itemize}
\end{definition}
It is clear that $\leq_1$ partially orders $\mathcal C_1$. It is also clear that $\mathcal C_1$ is non-empty since taking $\mathcal I$ to be the ideal of non-stationary subsets of $\pom(H(\Omega))$, $(\Lambda_1, \mathcal I) \in \mathcal C_1$.
The requirement that $A \s \Lambda_1$ for every $(A, \mathcal I) \in \mathcal C_1$ is needed in the following lemma, which we will need in the proof of the main lemma, Lemma~\ref{Shelahidea}.
\begin{lemma} \label{lemma44}Let $(A, \mathcal I) \in \mathcal C_1$ and let $k< \omega$. Then there is a $(B, \mathcal I)\leq_1 (A, \mathcal I)$ and $U \in \mathcal B$ such that for every $M \in B$, $U_{F(M), k} = U$.
\end{lemma}
\begin{proof} Let $\Phi: A \rightarrow \mathcal B$ be the function $M \mapsto U_{F(M), k}$. The nature of $F$ and Lemma~\ref{lemma42} ensure that $\Phi$ is a choice function on $A$, and of course $A \in \mathcal I^+$ and $\mathcal I$ is a normal ideal. So let $B \s A$ in $\mathcal I^+$ and $U \in \mathcal B$ be such that for every $M \in B$, $\Phi(M) = U$. 
\end{proof}
The requirement that $F(M) \notin M$ for $M \in \Lambda_1$ is needed for the following, which in turn will be needed in Lemma~\ref{weaksatij}.
\begin{lemma} \label{lemma45}Let $(A, \mathcal I) \in \mathcal C_1$ and let $x \in X$. 

Then $\{M \in A \mid F(M)= x\}\in \mathcal I$.
\end{lemma}
\begin{proof} Let $B:= \{M \in A \mid F(M)= x\}$ and suppose towards a contradiction that $B \in \mathcal I^+$. As $\mathcal I$ is a fine ideal, $\{M \in B \mid x \in M\}\in \mathcal I^+$. But if $M \in B$ then $F(M) =x$ so that $x \in \overline{M\cap X} \setminus M$ whereas also $x \in M$, a contradiction.
\end{proof}

\section{The construction of Raghavan and Todor{\v c}evi{\' c}}\label{sectionRT}
We recall some objects we have fixed in the previous section.
\begin{itemize}
\item $(X, \tau)$ is a regular topological space which is not left-separated and has a point-countable basis $\mathcal B$. 
\item $\Omega$ is an uncountable regular cardinal such that $X, \tau, \mathcal B \in H(\Omega)$.
\item $\langle U_{x,k}\mid x \in X, k< \omega\rangle$ is such that for every $x \in X$, $\langle U_{x, k} \mid k< \omega\rangle$ enumerates $\mathcal B_{\{x\}}$.
\item  $\Lambda_1 \s \pom(H(\Omega))$ is a fixed stationary subset of $\pom(H(\Omega))$ such that for every $M \in \Lambda_1$, $M \prec H(\Omega)$ and $X, \tau, \mathcal B \in M$, and $\overline{M \cap X} \setminus M \neq \emptyset$. 
\item $F: \Lambda_1 \rightarrow X$ is such that for every $M \in \Lambda_1$, $F(M) \in \overline{M \cap X} \setminus M$.
\end{itemize}
Additionally, fix $K$ a positive natural number and $c:[X]^2\rightarrow K$ a colouring of pairs from $X$. Also, let $\mathbb I$ be the collection of fine normal ideals on $\pom(H(\Omega))$.

The main result of this section is Theorem~\ref{RaghavanTodorcevicconstruction} which gives a sufficient condition for the existence of a subset of $X$ homeomorphic to $\mathbb Q$ on which $c$ takes at most $2$ colours. Again, the ideas here are all from Raghavan and Todor{\v c}evi{\'c}'s \cite{MR4190059} only the objects are different.

Recall that $\mathcal C_1$ has been defined earlier in Definition~\ref{defC1}. The following is a variant.
\begin{definition} \label{defC2}Let $\mathcal C_2$ denote the set of all triples $(A, B, \mathcal I)$ where 
\begin{itemize}
\item $\mathcal I \in \mathbb I$ and 
\item $A, B \s \Lambda_1$ and $A, B \in \mathcal I^+$. 
\end{itemize}
We define a relation $\leq_2$ on $\mathcal C_2$ as follows: for $(A, B, \mathcal I), (C, D, \mathcal J) \in \mathcal C_2$, 
$$(C, D, \mathcal J) \leq_2 (A, B, \mathcal I) \text{ if } C \subseteq A\ \&\ D \subseteq B\ \&\ \mathcal J\supseteq \mathcal I.$$
\end{definition}
Just as $\leq_1$ partially orders $\mathcal C_1$,  $\leq_2$ partially orders $\mathcal C_2$. And just as $\mathcal C_1$ is non-empty, $\mathcal C_2$ is also non-empty.

We remind the reader of our convention that if $c(x,y)$ is defined, that is, has a specific value, then we implicitly assume that also $x \neq y$.
\begin{definition}\label{defnweaksaturation} Let $(A, B, \mathcal I) \in \mathcal C_2$ and let $(i,j) \in K^2$. 
\begin{enumerate}
\item We say that $(A, B)$ is \emph{weakly $(i,j)$-saturated over $\mathcal I$} if 
$$\exists^{\mathcal I} M\in A\exists^{\mathcal I}N \in B\, [c(F(M),F(N))=i]$$
and 
$$\exists^{\mathcal I} N \in B \exists^{\mathcal I} M \in A \,[c(F(M),F(N)) = j].$$

    \item We say that $(A, B)$ is \emph{$(i, j)$-saturated over $\mathcal I$} if for every $C \s A$ and every $D \s B$ with $C, D \in \mathcal I^+$,  $(C,D)$ is weakly $(i,j)$-saturated over $\mathcal I$.
\end{enumerate}
\end{definition}
The following are the only properties of these notions that we will need in this section. In Secion~\ref{sectionstartingpoint} we will study them in more detail.
\begin{lemma}\label{firstbasicproperties} Let $(A, B, \mathcal I) \in \mathcal C_2$ and let $(i,j) \in K^2$. Suppose that $(A, B)$ is $(i,j)$-saturated over $\mathcal I$.
\begin{enumerate}
\item Then, 
$$\forall^{\mathcal I} M\in A\exists^{\mathcal I}N\in B\, [c(F(M),F(N))=i]$$
and 
$$\forall^{\mathcal I} N \in B \exists^{\mathcal I} M \in A \,[c(F(M),F(N)) = j].$$
\item For every $C \s A$ and $D \s B$ with $C, D \in \mathcal I^+$, also $(C, D)$ is $(i,j)$-saturated over $\mathcal I$.
\end{enumerate}
\end{lemma}
\begin{proof}
\begin{enumerate}
\item Suppose that the conclusion does not hold. Since the other case is symmetric, suppose that this is because it is not the case that 
$$\forall^{\mathcal I} M \in A \exists^{\mathcal I}N \in B\, [c(F(M),F(N)) = i].$$
So, let $C:=\{M \in A\mid \neg (\exists^{\mathcal I}N\in B\, [c(F(M),F(N)) = i])\}$, and then $C \in \mathcal I^+$. Now we have that $(C, B)$ is weakly $(i,j)$-saturated over $\mathcal I$. However, clearly 
$$\{M \in C \mid \exists^{\mathcal I}N \in B\, [c(F(M),F(N)) = i]\} = \emptyset,$$
 and so we reach a contradiction.
\item Easy.
\end{enumerate}
So the lemma is proved.
\end{proof}
Compare the following with \cite[Lemma~45]{MR4190059}.
\begin{definition} \label{winnerdefinition} Let $(A, \mathcal I) \in \mathcal C_1$.
\begin{enumerate}
\item Let $B \in \mathcal I^+$ with $B \s A$. Let $x \in X$ and  $(i,j) \in K^2$. We say that $x$ is an \emph{$(i, j)$-winner below $B$ over $\mathcal I$} if there is an $M \in B$ with $F(M) =x$ and there is a sequence $\langle T_k \mid k< \omega\rangle$ such that 
\begin{itemize}
\item for every $k<  \omega$, $T_k \s B$ and $T_k \in \mathcal I^+$;
    \item for every $N \in \bigcup_{k< \omega}T_k$, $F(N)\neq x$ and $c(x, F(N)) = i$;
    \item for every $k< l < \omega$, $(T_l, T_k)$ is $(i, j)$-saturated over $\mathcal I$;
    \item for every $k< \omega$, for every $N \in T_k$, $F(N) \in U_{x,k}$.
\end{itemize}
\item For $(i,j)\in K^2$, we say that $(A, \mathcal I)$ is an \emph{$(i, j)$-winning pair} if for every $B \s A$ with $B \in \mathcal I^+$, there is an $(i,j)$-winner below $B$ over $\mathcal I$.  
\item We say that $(A, \mathcal I)$ is a \emph{winning pair for $c$} if $(A,\mathcal I)$ is an $(i, j)$-winning pair for some $(i, j)\in K^2$.
\end{enumerate}

\end{definition}

\begin{theorem}[Raghavan-Todor{\v c}evi{\' c}]
\label{RaghavanTodorcevicconstruction}
Let $(A, \mathcal I) \in \mathcal C_1$ and suppose that $(A, \mathcal I)$ is a winning pair for $c$.

Then, there is a $Y \s X$ which is homeomorphic to $\mathbb Q$ and such that $|c``[Y]^2|\leq 2$.  
\end{theorem}
\begin{proof}
Suppose that $(i,j) \in K^2$ is such that  $(A, \mathcal I)$ is an $(i,j)$-winning pair.

Let $\vec \sigma = \langle \sigma_n\mid n< \omega\rangle$ be an injective enumeration of ${}^{< \omega}\omega$ satisfying that if $\rho_0 \sqsubset  \rho_1$ are elements of ${}^{< \omega}\omega$, then $\rho_0$ is enumerated before $\rho_1$. We will recursively select an injective sequence $\langle y_{\sigma_n} \mid n <\omega \rangle$ of elements of $X$ and the set of these elements will constitute the desired set $Y$. 

We will ensure that for $\rho \in {}^{< \omega}\omega$, the sequence $\langle y_{\rho{}^\smallfrown \langle k\rangle} \mid k< \omega\rangle $ is dense around $y_\rho$: for $k< \omega$, we will enusure that $y_{\rho {}^\smallfrown\langle k \rangle} \in U_{y_\rho, k}$. Once this is achieved, $Y$ would be a countable subset of a space with a point-countable basis and hence also second countable. As $X$ is regular Hausdorff, it would follow that $Y$ is metrisable by the Urysohn Metrisation Theorem, and as the construction would ensure that $Y$ is dense-in-itself, $Y$ would be homeomorphic to $\mathbb Q$ by the theorem of Sierpi{\'n}ski we have mentioned already. 

We will also secure that $c``[Y]^2= \{i,j\}$ (recall that $i$ and $j$ need not be distinct), and more informatively, we will have that for $m <n < \omega$, 
\begin{itemize}
\item  if $\sigma_m\sqsubset \sigma_n$, then $c(y_{\sigma_m}, y_{\sigma_n}) = i$, and
\item otherwise, if $\sigma_m <_{\mathrm{lex}} \sigma_n$ then $c(y_{\sigma_m}, y_{\sigma_n}) = j$, and 
\item otherwise, if $\sigma_n <_{\mathrm{lex}} \sigma_m$, then $c(y_{\sigma_m}, y_{\sigma_n}) = i$.
\end{itemize}
From hereon, to simplify the notation, for $n< \omega$ we will also refer to $y_{\sigma_n}$ as $y_n$.

We now give details about the recursive procedure. First we talk about the inductive assumption, then how to perpetuate it. For $n< \omega$, we will have at the end of stage $n$ the sequence $\langle y_m \mid m\leq n\rangle$, the portion of $Y$ constructed so far, and a subtree of ${}^{< \omega}\omega$ which we partition into two types of nodes, $\mathcal B_n$, the \emph{branching nodes}, and $\mathcal L_n$, the \emph{leaves}. The nodes of $\mathcal B_n$ will be the finite set of $\rho \in {}^{< \omega}\omega$ for which we have already decided $y_\rho$. 
The nodes of $\mathcal L_n$ will be the $\rho \in {}^{<\omega}\omega$ for which we have made some commitment to what $y_\rho$ is, but not decided it totally. So, we have also a labelling $h_n: \mathcal L_n \rightarrow \mathcal I^+$ which will give information about this commitment. Eventually any $ \rho \in \mathcal L_n$ will be added to $\mathcal B_{n'}$ for some $n' > n$ (we can be precise: the least such $n'< \omega$ is such that $\rho= \sigma_{n'}$) and then we will ensure that for some $M \in h_n(\rho)$, $F(M) = y_\rho$. Naturally, as the recursion progresses our approximations to $y_\rho$ will get better and better. 

For details, at the end of stage $n$ the objects $\langle y_m \mid m\leq n\rangle, \mathcal B_n, \mathcal L_n, h_n$ will satisfy the following properties.
\begin{enumerate}
\item (Injectivity requirement) The sequence $\langle y_m \mid m \leq n \rangle$ is an injective sequence of elements of $X$.
\item (Tree requirement) The set $\mathcal L_n \cup \mathcal B_n$ satisfies the following:
\begin{itemize}
\item the set $\mathcal L_n \cup \mathcal B_n$ is a subtree of  ${}^{< \omega}\omega$, that is, closed under initial segments;
\item the branching nodes $\mathcal B_n$ will be the set $\{\sigma_m \mid m\leq n\}$; 
\item for every $\rho\in \mathcal B_n$,  $\{\rho{}^\smallfrown\langle k\rangle \mid k< \omega\} \s \mathcal L_n \cup \mathcal B_n$;
\item the nodes of $\mathcal L_n$ are terminal nodes of the tree $\mathcal L_n \cup \mathcal B_n$: for $\rho\in \mathcal L_n$ and $\rho' \in \mathcal L_n \cup \mathcal B_n$, $\rho \not \sqsubset \rho'$.
\end{itemize}
\item (Convergence requirement) For $k< \omega$ such that $\rho, \rho{}^\smallfrown\langle k \rangle$ are both in $\mathcal B_n$, $y_{\rho{}^\smallfrown\langle k \rangle} \in U_{\rho, k}$.
\item (Ramsey requirement) The sequence $\langle y_m \mid m\leq n\rangle$ satisfies that for $m<m' \leq n$,
\begin{itemize}
\item  if $\sigma_{m } \sqsubset \sigma_{m' }$, then $c(y_{\sigma_{m}}, y_{\sigma_{m'}}) = i$, and
\item otherwise, if  $\sigma_{m }<_{\mathrm{lex}} \sigma_{m'}$, then $c(y_{\sigma_{m}}, y_{\sigma_{m'}}) = j$, and 
\item otherwise,  $\sigma_{m' }<_{\mathrm{lex}} \sigma_{m}$, and then $c(y_{\sigma_{m }}, y_{\sigma_{m'}}) = i$.
\end{itemize}
\item The labelling $h_n: \mathcal L_n \rightarrow \mathcal I^+$ satisfies the following requirements:
\begin{itemize}
\item for every $\rho \in \mathcal L_n$, $h_n(\rho) \s A$ and $h_n(\rho) \in \mathcal I^+$;
\item (injectivity promise) for every $\rho \in \mathcal L_n$, for every $N \in h_n(\rho)$,  $F(N) \notin \{y_m \mid m \leq n \} = \emptyset$;
\item (convergence promise) for every $\rho\in \mathcal B_n$, and every $k< \omega$ such that $\rho{}^{\smallfrown}\langle k \rangle \in \mathcal L_n$, and every $N \in  h_n(\rho{}^{\smallfrown}\langle k\rangle)$, $F(N) \in U_{y_\rho, k}$;
\item (first Ramsey inductive hypothesis) if $\rho_0 <_{\mathrm{lex}} \rho_1$ are elements of $\mathcal L_n$, then $( h_n(\rho_1), h_n(\rho_0))$ is $(i,j)$-saturated over $\mathcal I$;
\item (second Ramsey inductive hypothesis) if $\rho\in \mathcal B_n$ and $k< \omega$ and $\rho{}^\smallfrown \langle k \rangle \in \mathcal L_n$, then for every $N \in h_n(\rho{}^\smallfrown \langle k \rangle)$, $c(F(N), y_\rho) = i$;
\item (third Ramsey inductive hypothesis) if $\rho_0 \in \mathcal B_n$ and $\rho_1\in \mathcal L_n$ and $\rho_0 \not \sqsubset  \rho_1$ and $\rho_0 <_{\mathrm{lex}} \rho_1$, then  for every $N\in h_n(\rho_1)$, $c(F(N), y_{\rho_0}) = j$;
\item (fourth Ramsey inductive hypothesis) if $\rho_0 \in \mathcal B_n$ and $\rho_1 \in \mathcal L_n$ and $\rho_0 \not \sqsubset  \rho_1$ and $\rho_1 <_{\mathrm{lex}} \rho_0$, then  for every $N \in h_n(\rho_1)$, $c(F(N), y_{\rho_0}) = i$.
\end{itemize}
\end{enumerate}

$\br$ \underline{The base step.} To select $y_0$, let $x$ be an $(i,j)$-winner below $A$ over $\mathcal I$, which is witnessed by a sequence $\langle T_k \mid k< \omega\rangle$ of elements of $\mathcal I^+$ which are subsets of $A$. We let $y_0 := x$. We also take $\mathcal L_0:= \{\langle k \rangle \mid k< \omega\}$ and let $h_0: \mathcal L_0 \rightarrow \mathcal I^+$ be the function $\langle k\rangle \mapsto T_k$, and let $\mathcal B_0:= \{\langle \rangle\}$.

$\br$ \underline{The recursive step.} Suppose that $\langle y_m \mid m\leq n\rangle, \mathcal B_n, \mathcal L_n, h_n$ have been constructed satisfying the above properties. We describe how to select $\langle y_m \mid m\leq n+1\rangle, \mathcal B_{n+1}, \mathcal L_{n+1}, h_{n+1}$. Clearly we must have that 
$$\mathcal B_{n+1}:= \{\sigma_m \mid m \leq n+1\}.$$
Now note that by the nature of the enumeration $\vec \sigma$, for every $\rho \in {}^{< \omega}\omega$ such that $\rho\sqsubset  \sigma_{n+1}$, we have that $\rho \in \mathcal B_n$. So it follows that $\sigma_{n+1} \in \mathcal L_n$. Then by the nature of $\mathcal L_n$, for every $\rho\in \mathcal L_n$, $\rho \not \sqsubset  \sigma_{n+1}$ and $\sigma_{n+1} \not \sqsubset \rho$. So, taking
$$\mathcal L_{n+1} :=  (\mathcal L_n \cup\{\sigma_{n+1}{}^{\smallfrown}\langle k\rangle \mid k< \omega\}) \setminus \{\sigma_{n+1}\},$$ 
we see that the tree requirement at stage $n+1$ is taken care of.

In order to ensure continuity between the stages of the recursion, we will ensure that $y_{\sigma_{n+1}} = F(M)$ for some $M \in h_n(\sigma_{n+1})$ and also, for $\rho \in \mathcal L_n \setminus\{\sigma_{n+1}\}$, $h_{n+1}(\rho) \s h_n(\rho)$. We will refer to this as the \emph{continuity promise}.
This combined with the convergence promise at stage $n$ will immediately ensure that the convergence requirement at stage $n+1$ is met.

Let $B:= h_n(\sigma_{n+1})$, and for $\rho\in \mathcal L_n \setminus\{\sigma_{n+1}\}$, let $D_{\rho}:= h_n(\rho)$. 
Suppose that $\rho\in \mathcal L_n \setminus\{\sigma_{n+1}\}$ is such that $\sigma_{n+1}<_{\mathrm{lex}}\rho$. The first Ramsey inductive hypothesis at stage $n$ ensures that $(D_{\rho}, B)$ is $(i,j)$-saturated over $\mathcal I$. By Lemma~\ref{firstbasicproperties}(i), we have that 
$$\forall^{\mathcal I} M \in B\exists^{\mathcal I} N\in D_{\rho}\, [c(F(M),F(N)) = j].$$
Similarly, for $\rho\in \mathcal L_n \setminus\{\sigma_{n+1}\}$ such that $\rho<_{\mathrm{lex}}\sigma_{n+1}$, the first Ramsey inductive hypothesis at stage $n$ ensures that $(B, D_\rho)$ is $(i,j)$-saturated over $\mathcal I$. By Lemma~\ref{firstbasicproperties}(i), we have that 
$$\forall^{\mathcal I} M\in B\exists^{\mathcal I} N\in D_{\rho}\, [c(F(M), F(N)) = i].$$
So, using the fact that $\mathcal I$ is $\aleph_1$-complete, let $C\s B$ be an element of $\mathcal I^+$ such that for every $M \in C$, 
\begin{itemize}
\item if $\rho\in \mathcal L_n \setminus\{\sigma_{n+1}\}$ is such that $\sigma_{n+1}<_{\mathrm{lex}}\rho$, then 
$$\exists^{\mathcal I} N\in D_{\rho}\, [c(F(M),F(N)) = j],$$
\item and if $\rho\in \mathcal L_n \setminus\{\sigma_{n+1}\}$ is such that $\rho<_{\mathrm{lex}}\sigma_{n+1}$, then 
$$\exists^{\mathcal I} N\in D_{\rho}\, [c(F(M), F(N)) = i].$$
\end{itemize}
Now, $C \s B \s A$ and $C \in \mathcal I^+$, so by the hypotheses we have that there is an $(i,j)$-winner below $C$ over $\mathcal I$. So, fix $M \in C$ such that $F(M)$ is such an $(i,j)$-winner below $C$ over $\mathcal I$ and then let $y_{\sigma_{n+1}}:= F(M)$, and so by our notation, also $y_{n+1}:= F(M)$. 

It should be clear from the injectivity promise at stage $n$ that the sequence $\langle y_m \mid m \leq n+1\rangle$ is injective, so the injectivity requirement at stage $n+1$ is taken care of. The second, third, and fourth Ramsey inductive hypotheses at stage $n$ ensure that the Ramsey requirement is met at stage $n+1$. So we have taken care of the requirements (i)-(iv).

Since $y_{n+1}$ is an $(i, j)$-winner below $C$ over $\mathcal I$, fix a sequence $\langle T_k \mid k< \omega\rangle$ such that 
\begin{itemize}
\item for every $k<  \omega$, $T_k \s C$ and $T_k \in \mathcal I^+$;
    \item for every $N \in \bigcup_{k< \omega}T_k$, $F(N) \neq y_{n+1}$ and $c(y_{n+1},F(N)) = i$;
    \item for every $k< l < \omega$, $(T_l, T_k)$ is $(i, j)$-saturated over $\mathcal I$;
    \item for every $k< \omega$, for every $N \in T_k$, $F(N) \in U_{y_{n+1},k}$.
\end{itemize}
Now to define $h_{n+1}: \mathcal L_{n+1}\rightarrow J^+$.  First, $h_{n+1}\restriction \{\sigma_{n+1}{}^{\smallfrown}\langle k \rangle \mid k< \omega\}$ is just the function $\sigma_{n+1}{}^{\smallfrown}\langle k \rangle \mapsto T_k$. At this point note that for every $k< \omega$, 
$$h_{n+1}(\sigma_{n+1}{}^{\smallfrown}\langle k \rangle) = T_k \s C \s B  = h_n(\sigma_{n+1}).$$
This, the continuity promise, and the choice of $\langle T_k \mid k< \omega\rangle$ and the first Ramsey inductive hypothesis up to stage $n$ ensures that the first Ramsey inductive hypothesis continues at stage $n+1$. 

It should also be clear that the convergence promise has been taken care of by the continuity promise and the fact that $y_{n+1}$ is an $(i,j)$-winner below $C$ over $\mathcal I$ as witnessed by $\langle T_k\mid k< \omega\rangle$. The same reason tells us that the second Ramsey inductive hypothesis is taken care of.

Lastly, to describe $h_{n+1}\restriction (\mathcal L_n \setminus \{\sigma_{n+1}\})$. So, let $\rho \in  \mathcal L_n \setminus \{\sigma_{n+1}\}$. Recall that $D_{\rho}= h_n(\rho)$. By the choice of $C$, and since $M \in C$ and $y_{n+1} =F(M)$, we have that
\begin{itemize}
\item if $\sigma_{n+1} <_{\mathrm{lex}} \rho$, then for $\widetilde D_\rho:=\{N\in D_{\rho} \mid c(F(N),y_{n+1}) = j\}$, $\widetilde D_\rho \in \mathcal I^+$,
\item  if $\rho<_{\mathrm{lex}} \sigma_{n+1} $, then for $\widetilde D_\rho:=\{N \in D_{\rho} \mid c(F(N),y_{n+1}) = i\}$, $\widetilde D_\rho \in \mathcal I^+$.
\end{itemize}
Finally, for $\rho \in  \mathcal L_n \setminus \{\sigma_{n+1}\}$ we let $h_{n+1}(\rho):= \widetilde D_\rho$, and then clearly $h_{n+1}(\rho)\s h_n(\rho)$ and we have met the continuity promise. 

Now for the third and fourth Ramsey inductive hypotheses. It should be clear from the construction that we only need to verify them for $\sigma_{n+1} \in \mathcal B_{n+1}$ and $\rho \in \mathcal L_n \setminus \{\sigma_{n+1}\}$. However, the definition of $\widetilde D_\rho$ was precisely made to ensure this.

Lastly, we come to the injectivity promise. It should be clear from the construction that we only need to verify this for $\sigma_{n+1} \in \mathcal B_{n+1}$ and $\rho \in \mathcal L_n \setminus \{\sigma_{n+1}\}$. But clearly for every $N\in \widetilde D_\rho$ we have that $c(F(N),y_{n+1})$ is defined, so $y_{n+1} \neq F(N)$ for every $N \in \widetilde D_{\rho}$. 

With that the recursive step is completed, and so is the proof.
\end{proof}

\section{A decisive pair}\label{sectionstartingpoint}
The main result of this section is Lemma~\ref{firstmainlemma} which provides us with the decisive starting point for the main lemma, Lemma~\ref{Shelahidea}, in Section~\ref{sectionmainlemma}. The exhaustion argument in Lemma~\ref{findinga} is similar to one in Raghavan and Todor{\v c}evi{\' c} \cite{MR4190059} though our partial orders are different as they allow the ideal to vary. This is necessitated by the presence of two ideals in the statement of Lemma~\ref{firstmainlemma}. Normality of the ideals plays a key role in Lemma~\ref{onestep}. The notion of unsaturation (see Definition~\ref{defnunsaturation}) plays a more prominent role here than in \cite{MR4190059} for this reason.

Let us recall our collection of fixed objects:
\begin{itemize}
\item $(X, \tau)$ is a regular topological space which is not left-separated and has a point-countable basis $\mathcal B$. 
\item $K$ is a positive natural number and $c:[X]^2\rightarrow K$ is a colouring of pairs from $X$.
\item $\Omega$ is an uncountable regular cardinal such that $X, \tau, \mathcal B \in H(\Omega)$.
\item $\langle U_{x,k}\mid x \in X, k< \omega\rangle$ is such that for every $x \in X$, $\langle U_{x, k} \mid k< \omega\rangle$ enumerates $\mathcal B_{\{x\}}$.
\item  $\Lambda_1 \s \pom(H(\Omega))$ is a fixed stationary subset of $\pom(H(\Omega))$ such that for every $M \in \Lambda_1$, $M \prec H(\Omega)$ and $X, \tau, \mathcal B \in M$, and $\overline{M \cap X} \setminus M \neq \emptyset$. 
\item $F: \Lambda_1 \rightarrow X$ is such that for every $M \in \Lambda_1$, $F(M) \in \overline{M \cap X} \setminus M$.
\item $\mathbb I$ is the collection of fine normal ideals on $\pom(H(\Omega))$.
\end{itemize}

Recall that saturation and weak saturation were defined in Definition~\ref{defnweaksaturation}. The following is related to these.
\begin{definition} \label{defnunsaturation}
Let $(A, B, \mathcal I) \in \mathcal C_2$ and $F \s K^2$. We say that $(A, B)$ is \emph{$F$-unsaturated over $\mathcal I$} if for every $(i,j) \in F$, $(A, B)$ is not weakly $(i,j)$-saturated over $\mathcal I$.
\end{definition}

For the next lemma we again remind the reader of our convention that if $c(x,y)$ is defined, that is, has a specific value, then we implicitly assume that also $x \neq y$.
\begin{lemma}\label{weaksatij}
For any $(A, B,\mathcal I)\in \mathcal C_2$, for some $(i,j) \in K^2$ we have that $(A,B)$ is weakly $(i,j)$-saturated over $\mathcal I$.
\end{lemma}
\begin{proof} 
Fix for the moment an $M \in A$, and let $x := F(M)$. By Lemma~\ref{lemma45}, $\{N \in B \mid F(N) = x\} \in \mathcal I$. So $C \in \mathcal I^+$ where 
$$C:=\{N \in B \mid F(N) \neq x\}.$$
So for some $i_M<K$ we have that 
    $$\{N\in C \mid c(x,F(N))= i_M\} \in \mathcal I^+.$$
To summarise, for every $M \in A$, for some $i_M< K$ we have that 
    $$\{N\in B \mid F(M) \neq F(N)\ \&\ c(F(M),F(N))= i_M\} \in \mathcal I^+.$$
So stabilising on $i_M$ we obtain an $i< K$ such that 
    $$\{M \in A \mid i_M = i\} \in \mathcal I^+.$$
    So, 
    $$\exists^{\mathcal I} M\in A\exists^{\mathcal I}N \in B\,[c(F(M),F(N))=i].$$
Similarly we can find a $j< K$ such that 
$$\exists^{\mathcal I} N \in B \exists^{\mathcal I} M \in A\, [c(F(M), F(N)) = j],$$
and conclude the proof.
\end{proof}

\begin{lemma}\label{basicproperties} Let $(C, D, \mathcal J) \leq_2 (A, B, \mathcal I)$ be elements of $\mathcal C_2$. Let $(i,j) \in K^2$. 

If $(A, B)$ is $\{(i,j)\}$-unsaturated over $\mathcal I$, then also $(C, D)$ is $\{(i,j)\}$-unsaturated over $\mathcal J$.
\end{lemma}
\begin{proof}  As $\mathcal I\subseteq \mathcal J$, we have that $\mathcal I^+ \supseteq \mathcal J^+$. Suppose that $(A, B)$ is not weakly $(i,j)$-saturated over $\mathcal I$. There are two possibilities for why. The first is that 
$$\{M\in A \mid \exists^{\mathcal I} N \in B\, [c(F(M),F(N)) =i]\} \in \mathcal I,$$
in which case as $C \s A$ and $D \s B$  and $\mathcal J^+ \s \mathcal I^+$, also 
$$\{M\in C \mid \exists^{\mathcal J} N\in D\, [c(F(M), F(N)) =i]\} \in \mathcal J,$$
so that $(C, D)$ is not weakly $(i,j)$-saturated over $\mathcal J$ as well. 

The other possibility is that 
$$\{N\in B \mid \exists^{\mathcal I} M \in A\, [c(F(M), F(N)) =j]\} \in \mathcal I,$$
in which case we similarly have that 
$$\{N\in D \mid \exists^{\mathcal J} M \in C\, [c(F(M), F(N)) =j]\} \in \mathcal J,$$
so again $(C, D)$ is not weakly $(i,j)$-saturated over $\mathcal J$.
\end{proof}

\begin{lemma} \label{findinga}There is an $(A, \mathcal I) \in \mathcal C_1$ and an $F \subsetneq K^2$ and a pair $(i,j) \in K^2 \setminus F$ such that the following hold:  
\begin{enumerate}
\item for every $(C, D, \mathcal J) \leq_2 (A, A, \mathcal I)$ from $\mathcal C_2$, if $(C, D)$ is $F$-unsaturated over $\mathcal J$, then $(C, D)$ is $(i,j)$-saturated over $\mathcal J$;
\item for every $(B, \mathcal J_0)\leq_1 (A, \mathcal I)$, there is a $(C, D, \mathcal J_1) \leq_2(B, B, \mathcal J_0)$ such that $(C, D)$ is $F$-unsaturated over $\mathcal J_1$.
\end{enumerate} 
\end{lemma}
\begin{proof}Below, we will use the letters $a,b$ for members of $\mathcal C_1$, and $p,q$ for members of $\mathcal C_2$. 
We define a function $\Sigma_0: \mathcal C_2 \rightarrow \p(K^2)$ as follows: for $(A, B, \mathcal I) \in \mathcal C_2$,
$$\Sigma_0(A, B, \mathcal I):= \{(i,j) \in K^2\mid (A,B)\text{ is not weakly $(i,j)$-saturated over }\mathcal I\}.$$
Note that by Lemma~\ref{weaksatij}, $K^2 \notin \im(\Sigma_0)$. Also, for $q \leq_2 p$ elements of $\mathcal C_2$,  $\Sigma_0(q) \supseteq \Sigma_0(p)$ by Lemma~\ref{basicproperties}.

For $q\in \mathcal C_2$, we will say that it is \emph{maximal for $\Sigma_0$} if for every $q' \leq_2 q$ from $\mathcal C_2$, we have that 
$$\Sigma_0(q) = \Sigma_0(q').$$

\begin{claim}\label{claim382}For every $p \in \mathcal C_2$, there is at least one $q \leq_2 p$ which is maximal for $\Sigma_0$.
\end{claim}
\begin{why} The map $\Sigma_0: \mathcal C_2 \rightarrow \p(K^2)$ has the property that for $q' \leq_2 q$ from $\mathcal C_2$, we have that $\Sigma_0(q') \supseteq \Sigma_0(q)$. Since the set $K^2$ is finite, we cannot get an infinite $\leq_2$-decreasing sequence of elements of $\mathcal C_2$ whose $\Sigma_0$-values are strictly $\s$-increasing. The claim follows. 
\end{why}
Now, define a function $\Sigma_1: \mathcal C_2 \rightarrow \p(\p(K^2))$ as follows: for $p \in \mathcal C_2$ and $F \s K^2$, $F \in  \Sigma_1 (p)$ if there is a $q\in \mathcal C_2$ such that 
\begin{itemize}
\item $q \leq_2 p$,
\item $\Sigma_0(q) = F$, 
\item $q$ is maximal for $\Sigma_0$.
\end{itemize}
Note that for $p \in \mathcal C_2$, $K^2 \notin \Sigma_1(p)$.
Now we define a function $\Gamma: \mathcal C_1 \rightarrow  \p(\p(K^2))$ as follows: for $(A,\mathcal I) \in \mathcal C_1$, 
$$\Gamma(A, \mathcal I):=\Sigma_1(A, A ,\mathcal I).$$
Note that by Claim~\ref{claim382}, for every $a \in \mathcal C_1$, $\Gamma(a) \neq \emptyset$, and as we have pointed out before, $K^2 \notin \Gamma(a)$.
\begin{claim} Let $b \leq_1 a$ be elements of $\mathcal C_1$. Then $\Gamma(b) \subseteq \Gamma(a)$
\end{claim}
\begin{why}
Suppose that $b= (B, \mathcal J)$ and $a= (A, \mathcal I)$. Then if $p \leq_2(B, B, \mathcal J)$ and is maximal for $\Sigma_0$, then $p \leq_2(A, A, \mathcal I)$ as well, and naturally $p$ is still maximal for $\Sigma_0$.
\end{why}
\begin{claim} There is an $a \in \mathcal C_1$ such that for every $b \leq_1 a$ in $\mathcal C_1$, $\Gamma(b)= \Gamma(a)$.
\end{claim}
\begin{why}
The map $\Gamma: \mathcal C_1 \rightarrow  \p(\p(K^2))$ has the property that for $b\leq_1 a$ from $\mathcal C_1$, we have that $\Gamma(b) \subseteq \Gamma(a)$. Since the set $\p(\p(K^2))$ is finite, we cannot get an infinite $\leq_1$-decreasing sequence of elements of $\mathcal C_1$ whose $\Gamma$-values are strictly $\s$-decreasing. 
\end{why}
From now on we fix such an $a= (A, \mathcal I)$ in $\mathcal C_1$. Then $\Gamma(a) \s \p(K^2)$ and $\Gamma(a) \neq \emptyset$ and $K^2 \notin \Gamma(a)$.

Fix now $F \in \Gamma(a)$ such that for every $G \in \Gamma(a)$, if $F \s G$, then $F = G$. As $F \neq K^2$, we lastly fix $(i,j) \in K^2 \setminus F$. 
We will now show that $a$ and  $F$ and  $(i,j)$ are as required by the lemma. 
\begin{claim} Let $(C, D, \mathcal J) \in \mathcal C_2$ be such that $(C, D, \mathcal J) \leq_2(A, A, \mathcal I)$. Then, if $(C, D)$ is $F$-unsaturated over $\mathcal J$, then $(C, D)$ is $(i,j)$-saturated over $\mathcal J$.
\end{claim} 
\begin{why}Indeed, suppose that $(C, D)$ is $F$-unsaturated over $\mathcal J$ but not $(i,j)$-saturated over $\mathcal J$. In this case, there is $(C', D', \mathcal J)\leq_2 (C, D, \mathcal J)$ such that $(C', D')$ is not weakly $(i,j)$-saturated over $\mathcal J$, so that $(C', D')$ is $(F \cup\{(i,j)\})$-unsaturated over $\mathcal J$. Now, let $q \leq_2 (C', D', \mathcal J)$ be maximal for $\Sigma_0$. Then $q \leq_2 (A, A, \mathcal I)$ and $\Sigma_0(q) \supseteq F \cup\{(i,j)\}$. So let $G:= \Sigma_0(q)$. Then $G \in \Gamma(a)$ and $F \subsetneq G$. This contradicts the choice of $F$.
\end{why}
\begin{claim} Let $(B, \mathcal J_0)\leq_1 (A, \mathcal I)$. Then there is $(C, D, \mathcal J_1) \leq_2(B, B, \mathcal J_0)$ in $\mathcal C_2$ such that $(C, D)$ is $F$-unsaturated over $\mathcal J_1$.
\end{claim}
\begin{why} By the choice of $(A, \mathcal I)$, we have that $\Gamma(B, \mathcal J_0) = \Gamma(A, \mathcal I)$, so $F \in \Gamma(B, \mathcal J_0)$, so $F \in \Sigma_1(B, B, \mathcal J_0)$. So there is $(C, D, \mathcal J_1)\leq_2(B, B, \mathcal J_0)$ which is maximal for $\Sigma_0$ and such that $\Sigma_0(C, D, \mathcal J_1) = F$. So $(C, D)$ is $F$-unsaturated over $\mathcal J_1$, as desired.
\end{why}
This completes the proof.
 \end{proof}
\begin{lemma}\label{onestep}
Let $\mathcal I\s \mathcal J$ be elements of $\mathbb I$. Let $(C, D, \mathcal J) \in \mathcal C_2$ and $F \s K^2$ be such that $(C, D)$ is $F$-unsaturated over $\mathcal J$. 

Then, there is $(S, T, \mathcal J) \leq_2 (C, D, \mathcal J)$  such that $(S, T)$ is $F$-unsaturated over $\mathcal I$.
\end{lemma}
\begin{proof} Note that as $\mathcal I \s \mathcal J$, we have that $\mathcal J^+ \s \mathcal I^+$. We will prove the lemma by induction on the size of $F$. So, suppose first that $F = \{(i,j)\}$. So we have that $(C, D)$ is not weakly $(i,j)$-saturated over $\mathcal J$. Since the other case is symmetric, let us suppose that this is because 
$$\neg (\exists^{\mathcal J} M \in C \exists^{\mathcal J} N \in D\, [c(F(M), F(N)) = i]).$$
That is, the following set is in $\mathcal J$:
$$\{M \in C \mid \{N\in D \mid c(F(M), F(N)) = i\} \in \mathcal J^+\}.$$
So, $S \in (\mathcal J+C)^*$, where
$$S:=\{M \in C \mid \{N \in D \mid c(F(M), F(N))= i \} \in \mathcal J \}.$$
Note that
$$S=\{M \in C \mid \{N \in D \mid c(F(M), F(N))\neq i \} \in (\mathcal J+D)^* \}.$$
We remind the reader at this point: $c(F(M),F(N)) =i$ implies that $F(M) \neq F(N)$, but $c(F(M),F(N)) \neq i$ may also be for the reason that $F(M) = F(N)$. 

Let 
$$H':=\{x \in H \mid \exists M \in S\, [F(M) = x]\}.$$
For every $x\in  H'$, let 
$$T_x:=\{N \in D \mid c(x,F(N))\neq i \},$$ 
so for every $x \in H'$, we have that $T_x \in (\mathcal J+D)^*$. Now let $T:= \diagonalintersection_{x \in H'}T_x$, which is in $(\mathcal J+D)^*$ as $\mathcal J+D$ is normal. Since $C \in \mathcal J^+$ and $S\in (\mathcal J+C)^*$, we have that $S \in \mathcal J^+$ and hence also $S \in \mathcal I^+$. Similarly, $T \in \mathcal J^+$ and hence also $T \in \mathcal I^+$. 

Now for every $M \in S$, 
$T_{F(M)}= \{N \in D\mid c(F(M), F(N)) \neq i\}$,
and $T_{F(M)} \in (\mathcal I+T)^*$ as $T \in (\mathcal I+T)^*$ and by the definition of the diagonal intersection, see Fact~\ref{lemmaidealfacts}(vii). 

So, for every $M \in S$, $T \cap T_{F(M)} \in (\mathcal I+T)^*$, so 
$$\{N\in T \mid c(F(M),F(N)) \neq i\} \in (\mathcal I+T)^*.$$
So, for every $M \in S$, $\forall^{\mathcal I} N \in T\, [c(F(M), F(N)) \neq i]$. So $(S, T, \mathcal J) \leq_2 (C, D, \mathcal J)$ and $(S, T)$ is not weakly $(i,j)$-saturated over $\mathcal I$. This completes the base step.

Now, for the inductive step. Suppose that $|F| > 1$ and $(i,j) \in F$. First appealing to the inductive hypotheses, we obtain $(S', T', \mathcal J) \leq_2 (C, D, \mathcal J)$ such that $(S', T')$ is $(F \setminus\{(i,j)\})$-unsaturated over $\mathcal I$. By Lemma~\ref{basicproperties} we have that $(S', T')$ is not weakly $(i,j)$-saturated over $\mathcal J$. Then we can appeal to the base case applied to $(S', T', \mathcal J)$ to obtain $(S, T,\mathcal J) \leq_2 (S', T', \mathcal J)$ and $(S, T)$ is not weakly $(i,j)$-saturated over $\mathcal I$, and hence $(S, T)$ is $F$-unsaturated over $\mathcal I$ by Lemma~\ref{basicproperties}.
\end{proof}
\begin{remark} The above lemma shows that for $\mathcal I \s \mathcal J$ ideals in $\mathbb I$ and $(A, B, \mathcal J) \in \mathcal C_2$ and $(i,j)\in K^2$, if $(A, B)$ is $(i,j)$-saturated over $\mathcal I$ then also $(A, B)$ is $(i,j)$-saturated over $\mathcal J$.
\end{remark}

\begin{lemma}\label{firstmainlemma} There is a pair $(A, \mathcal I) \in \mathcal C_1$ and $(i,j) \in K^2$ such that 
\begin{itemize}
\item for every $\mathcal J\in \mathbb I$ with $\mathcal I \subseteq \mathcal J$, 
\item for every $B \s A$ with $B \in \mathcal J^+$,
\end{itemize}
there are $S, T \s B$ with $S, T\in \mathcal J^+$ such that $(S, T)$ is $(i,j)$-saturated over $\mathcal I$.
\end{lemma}
\begin{proof}
Let $A$, $\mathcal I$, $F$, and $(i,j)$ be as in Lemma~\ref{findinga}. Then for every $(B, \mathcal J) \leq_1 (A, \mathcal I)$, for some $(C, D, \mathcal J_1) \leq_2 (B, B, \mathcal J)$, we have that $(C, D)$ is $F$-unsaturated over $\mathcal J_1$. By Lemma~\ref{onestep}, we can find $(S, T, \mathcal J_1) \leq_2 (C, D, \mathcal J_1)$ such that $(S,T)$ is $F$-unsaturated over $\mathcal I$, and hence, also $(S, T)$ is $(i,j)$-saturated over $\mathcal I$. Since $S,T \in (\mathcal J_1)^+$ and $\mathcal J \s \mathcal J_1$, also $S,T \in \mathcal J^+$ and so the proof is completed. 
\end{proof}

\section{The main lemma}\label{sectionmainlemma}
We prove now our main lemma, which as we have mentioned earlier uses an idea we learned from Shelah's \cite{Sh:881}.

Let us recall again the objects we have fixed:
\begin{itemize}
\item $(X, \tau)$ is a regular topological space which is not left-separated and has a point-countable basis $\mathcal B$. 
\item $K$ is a positive natural number and $c:[X]^2\rightarrow K$ is a colouring of pairs from $X$.
\item $\Omega$ is an uncountable regular cardinal such that $X, \tau, \mathcal B \in H(\Omega)$.
\item $\langle U_{x,k}\mid x \in X, k< \omega\rangle$ is such that for every $x \in X$, $\langle U_{x, k} \mid k< \omega\rangle$ enumerates $\mathcal B_{\{x\}}$.
\item  $\Lambda_1$ is a fixed stationary subset of $\pom(H(\Omega))$ such that for every $M \in \Lambda_1$, $M \prec H(\Omega)$ and $X, \tau, \mathcal B \in M$, and $\overline{M \cap X} \setminus M \neq \emptyset$. 
\item $F: \Lambda_1 \rightarrow X$ is such that for every $M \in \Lambda_1$, $F(M) \in \overline{M \cap X} \setminus M$.
\item $\mathbb I$ is the collection of fine normal ideals on $\pom(H(\Omega))$.
\end{itemize}

\begin{lemma}\label{Shelahidea}
There is an $\mathcal I \in \mathbb I$ and an $A \in \mathcal I^+$ such that $(A, \mathcal I)$ is a winning pair for $c$.     
\end{lemma}
\begin{proof} 
Using Lemma~\ref{firstmainlemma}, let $(A, \mathcal I) \in \mathcal C_1$ and $(i,j) \in K^2$ be such that 
\begin{itemize}
\item for every $\mathcal J\in \mathbb I$ with $\mathcal I \subseteq \mathcal J$,
\item for every $B \subseteq A$ with $B \in \mathcal J^+$,
\end{itemize}
there are $S, T \s B$ with $S, T \in \mathcal J^+$ such that $(S, T)$ is $(i,j)$-saturated over $\mathcal I$.

We will prove that $(A, \mathcal I)$ is an $(i,j)$-winning pair. Let us recall from Definition~\ref{winnerdefinition} what this requires. We need to show that for every $B \in \mathcal I^+$ with $B \s A$, there is an $(i,j)$-winner $x$ below $B$ over $\mathcal I$.  This in turn requires us to show that there is an $M \in B$ with $F(M) =x$ and there is a sequence $\langle T_k \mid k< \omega\rangle$ such that 
\begin{itemize}
\item for every $k<  \omega$, $T_k \s B$ and $T_k \in \mathcal I^+$;
    \item for every $N \in \bigcup_{k< \omega}T_k$, $F(N)\neq x$ and $c(x, F(N)) = i$;
    \item for every $k< l < \omega$, $(T_l, T_k)$ is $(i, j)$-saturated over $\mathcal I$;
    \item for every $k< \omega$, for every $N \in T_k$, $F(N) \in U_{x,k}$.
\end{itemize}
So, fix $B \in \mathcal I^+$ with $B \s A$.
We want to show that there is an $(i,j)$-winner below $B$ over $\mathcal I$, so suppose towards a contradiction that there is no such element.

Let $\mathcal T$ be the set of all finite sequences $\sigma$ such that for some $l< \omega$, 
$$\sigma= \langle (S_k,T_k)\mid k<l \rangle,$$ 
and such that 
\begin{enumerate}
    \item for all $k< l$, $S_k, T_k \in \mathcal I^+$ with $S_k \cup T_k \s B$;
    \item for all $k< l$, $(S_k, T_k)$ is $(i,j)$-saturated over $ \mathcal I$;
    \item for all $0< k< l$, $S_k \cup T_k \s S_{k-1}$;
    \item for all $k< l$, there is a $U\in \mathcal B$ such that for every $M \in S_k \cup T_k$, $U= U_{F(M), k}$;
    \item for all $k< l$ and every $M\in S_k$, $\{N \in T_k \mid c(F(M), F(N)) = i\} \in  \mathcal I^+$.
\end{enumerate}
Note that $\mathcal T$ is closed under initial segments, so it is a tree of finite sequences. 

For every $M \in B$, let $\mathcal T^M$ consist of those $\sigma \in \mathcal T$ such that if $\sigma = \langle (S_k, T_k) \mid k< l \rangle$ for some $l< \omega$, then for every $k< l$, $M\in S_k$. Note that $\mathcal T^M$ is also closed under initial segments, so it is a subtree of $\mathcal T$. In particular for every $M \in B$, the empty sequence $\langle \rangle$ is in $\mathcal T^M$.
\begin{claim}
For every $M \in B$, $\mathcal T^M$ is well-founded.    
\end{claim}
\begin{why}
Suppose that this is not so, so let $M \in B$ be such that we can find an infinite sequence 
$$b= \langle (S_k, T_k)\mid k< \omega \rangle$$
such that all of the initial segments of $b$ are in $\mathcal T^M$. 
It follows that $M \in \bigcap_{k< \omega} S_k$. Let $x:= F(M)$. Now for every $k< \omega$, let 
$$\widetilde T_k:= \{N \in T_k\mid c(x, F(N))= i\}.$$
Then for every $k< \omega$ we have that  $\widetilde T_k\in \mathcal I^+$ and $\widetilde T_k\s T_k \s B$, and also, for every $N \in \widetilde T_k$, $x \neq F(N)$ and $c(x, F(N))=i$.

Also, for $k< l< \omega$, we have that 
$$\widetilde T_l\s T_l\s S_{l-1}\s S_k,$$
and then as $(S_k,T_k)$ is $(i,j)$-saturated over $\mathcal I$, also $ (\widetilde T_l,\widetilde T_k)$ is $(i,j)$-saturated over $\mathcal I$.

Lastly, for $k< \omega$,  as $\{M\} \cup \widetilde T_k \s S_k\cup T_k$, for every  $N \in \widetilde T_k$ we have that $U_{F(N), k} = U_{F(M), k}= U_{x,k}$, and hence also, $F(N) \in U_{x,k}$. 

Altogether, we have that $x$ is an $(i,j)$-winner below $B$ over $\mathcal I$ as witnessed by $M$ and the sequence $\langle \widetilde T_k \mid k< \omega\rangle$. This is a contradiction.

\end{why}
For every $M \in B$ the tree $\mathcal T^M$ is well-founded, so we define a rank $\rk^M: \mathcal T^M\rightarrow \ord$ via: for $\sigma \in \mathcal T^M$, 
$$\rk^M(\sigma):= \sup\{\rk^M(\sigma')+1 \mid \sigma' \in \mathcal T^M, \, \sigma \sqsubset  \sigma'\}.$$

We now get a chance to use a beautiful idea of Shelah \cite{Sh:881}. 
Let $\mathbf P$ be the set of all \emph{creatures} $\mathfrak p$ where a creature is a triple $\mathfrak p= (\sigma, Z, f)$ such that
\begin{itemize}
    \item $\sigma \in \mathcal T$;
    \item $Z \in \mathcal I^+$ with $Z \s B$, and more specifically, if $|\sigma| = 0$ then $Z:= B$, and if $\sigma = \langle (S_k,T_k) \mid k< l \rangle$ for some $0< l< \omega$, then $Z:= S_{l-1}$;
    \item $f \in {}^Z\ord$ is the function $M\mapsto \rk^M(\sigma)$, which is legitimate as $\sigma \in \mathcal T^M$ for every $M\in Z$.
\end{itemize}
Note that $\mathbf P$ is non-empty, since $(\langle\rangle, B, f) \in \mathbf P$ where $f \in {}^B \ord$ is the function $M \mapsto \rk^M(\langle \rangle)$. 

Now, given a creature $\mathfrak p= (\sigma, Z, f)$ in $\mathbf P$, as $\mathcal I+Z$ is an $\aleph_1$-complete ideal and $Z \in (\mathcal I+Z)^*$ and $f \in {}^Z\ord$, the rank $\rk_{\mathcal I+Z}(f)$ is well-defined, and we will call this the \emph{hybrid rank} of $\mathfrak p$ to avoid confusion among all the ranks.\footnote{The terminology is due to Assaf Rinot from \cite{safblog}.}

So, among all the creatures in $\mathbf P$, let $\mathfrak p= (\sigma, Z, f)$ be chosen such that its hybrid rank is minimal. Let $k:= |\sigma|$.

Now consider the ideal $\mathcal J:= \mathbf I[f, \mathcal I+Z]$. 
By Lemma~\ref{JDFfacts},
$$\mathcal I \s \mathcal I+Z \s \mathcal J,$$
and $\mathcal J$ is a fine normal ideal on $\pom(H(\Omega))$ so $\mathcal J\in \mathbb I$. 
As $Z \in \mathcal J^*$,  or more relevant, as $(Z, \mathcal J) \in \mathcal C_1$, by Lemma~\ref{lemma44} there is $U\in \mathcal B$ and $(Y,\mathcal J) \leq_1 (Z, \mathcal J)$ such that for every $M \in Y$, $U_{F(M),k} = U$. Now, $(Y, \mathcal J) \leq_1 (A, \mathcal I)$, and by the choice of $(A, \mathcal I)$ we can find $S', T \s Y$ such that $S', T \in \mathcal J^+$ and $(S', T)$ is $(i,j)$-saturated over $\mathcal I$. 
We have by Lemma~\ref{firstbasicproperties} that 
$$\forall^{\mathcal I} M\in S' \exists^{\mathcal I} N \in T [c(F(M), F(N)) =i].$$
So, let $S \s S'$ be such that $S' \setminus S\in \mathcal I$ and 
$$\forall M\in S\exists^{\mathcal I} N \in T [c(F(M), F(N)) =i].$$
Since $\mathcal I \s \mathcal J$ and $S' \in \mathcal J^+$, note that $S \in \mathcal J^+$ too, and of course $(S, T)$ is also $(i,j)$-saturated over $\mathcal I$. 

Let $\widetilde \sigma:= \sigma {}^{\smallfrown}\langle (S, T)\rangle$, and then we have that $\widetilde \sigma \in \mathcal T$.
Now, let $\widetilde {\mathfrak p}:= (\widetilde \sigma, \widetilde Z, \widetilde f)$, where
\begin{itemize}
    \item $\widetilde Z:= S$, and 
\item $\widetilde f \in {}^{S}\ord$ is the function $M\mapsto \rk^M(\widetilde \sigma)$.
\end{itemize}
Clearly $\widetilde{\mathfrak p} \in \mathbf P$. Now, let us compare $\rk_{\mathcal I+Z}(f)$ and $\rk_{\mathcal I+\widetilde Z}(\widetilde f)$. 
As $S \in \mathcal J^+ =( \mathbf I[f,\mathcal I+Z])^+$ and $S \s Z$, we have that
$$\rk_{\mathcal I+Z}(f) = \rk_{(\mathcal I+Z)+S}(f)  = \rk_{\mathcal I+S}(f) > \rk_{\mathcal I+S}(\widetilde f) = \rk_{\mathcal I+\widetilde Z}(\widetilde f).$$

To see why the inequality is true, note that for every $M \in S$, as $\sigma, \widetilde \sigma \in \mathcal T^M$ and $\sigma \sqsubset \widetilde \sigma$, 
$$f(M) = \rk^M(\sigma) > \rk^M(\widetilde \sigma) = \widetilde f(M),$$
so $\widetilde f <_{\mathcal I+S} f$. 
So we have that the hybrid rank of $\widetilde {\mathfrak p}$ is strictly less than the hybrid rank of $\mathfrak p$. This contradicts the choice of $\mathfrak p$ and we finish.
\end{proof}

We are now in a position to prove Theorem~D.
\begin{corollary} \label{maincorollary} Let $(X, \tau)$ be a regular topological space which is not left-separated and has a point-countable basis. Let $K$ be a positive natural number and $c:[X]^2 \rightarrow K$ a colouring of pairs of elements of $X$. Then there is a $Y\s X$ homeomorphic to $\mathbb Q$ such that $|c``[Y]^2| \leq 2$.
\end{corollary}
\begin{proof} By Lemma~\ref{Shelahidea}, there is an $\mathcal I\in \mathbb I$ and an $A \in \mathcal I^+$ such that $(A, \mathcal I)$ is a winning pair for $c$. By Theorem~\ref{RaghavanTodorcevicconstruction} this gives the existence of a $Y\s X$ as required.
\end{proof}

\section*{Acknowledgements}
This work was supported by the Israel Science Foundation (grant agreement 665/20). 
The result was presented in the Set Theory Seminar of Bar-Ilan University on 28\textsuperscript{th} March 2024. 
I would like to thank the participants for this opportunity and their patience. 
I would like to thank Assaf Rinot for many suggestions which have improved the article. 
I would like to thank Todd Eisworth for his interest in the `very weak' precipitous game: this will appear in the final version.


\begin{thebibliography}{10}

\bibitem{MR867644}
{\sc Baumgartner, J.~E.}
\newblock Partition relations for countable topological spaces.
\newblock {\em J. Combin. Theory Ser. A 43}, 2 (1986), 178--195.

\bibitem{MR2628717}
{\sc Devlin, D.~C.}
\newblock {\em Some {P}artition {T}heorems {A}nd {U}ltrafilters {O}n {O}mega}.
\newblock ProQuest LLC, Ann Arbor, MI, 1980.
\newblock Thesis (Ph.D.)--Dartmouth College.

\bibitem{MR4680287}
{\sc Dobrinen, N.}
\newblock Ramsey theory of homogeneous structures: current trends and open problems.
\newblock In {\em I{CM}---{I}nternational {C}ongress of {M}athematicians. {V}ol. {III}. {S}ections 1--4}. EMS Press, Berlin, [2023] \copyright 2023, pp.~1462--1486.

\bibitem{MR4253680}
{\sc Dobrinen, N., and Gasarch, W.}
\newblock When {R}amsey {T}heory fails settle for more colors (big {R}amsey degrees!).
\newblock {\em ACM SIGACT News 51}, 4 (2020), 30--46.

\bibitem{eisworth2023galvins}
{\sc Eisworth, T.}
\newblock Galvin's conjecture and weakly precipitous ideals, 2023.

\bibitem{MR280381}
{\sc Erd\H{o}s, P., and Hajnal, A.}
\newblock Unsolved problems in set theory.
\newblock In {\em Axiomatic {S}et {T}heory ({P}roc. {S}ympos. {P}ure {M}ath., {V}ol. {XIII}, {P}art {I}, {U}niv. {C}alifornia, {L}os {A}ngeles, {C}alif., 1967)\/} (1971), Proc. Sympos. Pure Math., XIII, Part I, Amer. Math. Soc., Providence, RI, pp.~17--48.

\bibitem{MR357122}
{\sc Erd\H{o}s, P., and Hajnal, A.}
\newblock Unsolved and solved problems in set theory.
\newblock In {\em Proceedings of the {T}arski {S}ymposium ({P}roc. {S}ympos. {P}ure {M}ath., {V}ol. {XXV}, {U}niv. {C}alifornia, {B}erkeley, {C}alif., 1971)\/} (1974), Proc. Sympos. Pure Math., Vol. XXV, Published for the Association for Symbolic Logic by the American Mathematical Society, Providence, RI, pp.~269--287.

\bibitem{MR202613}
{\sc Erd\H{o}s, P., Hajnal, A., and Rado, R.}
\newblock Partition relations for cardinal numbers.
\newblock {\em Acta Math. Acad. Sci. Hungar. 16\/} (1965), 93--196.

\bibitem{MR65615}
{\sc Erd\H{o}s, P., and Rado, R.}
\newblock Combinatorial theorems on classifications of subsets of a given set.
\newblock {\em Proc. London Math. Soc. (3) 2\/} (1952), 417--439.

\bibitem{MR167422}
{\sc Erd\H{o}s, P., and Tarski, A.}
\newblock On some problems involving inaccessible cardinals.
\newblock In {\em Essays on the foundations of mathematics}. Magnes Press, The Hebrew University, Jerusalem, 1961, pp.~50--82.

\bibitem{MR825729}
{\sc Fleissner, W.~G.}
\newblock Left separated spaces with point-countable bases.
\newblock {\em Trans. Amer. Math. Soc. 294}, 2 (1986), 665--677.

\bibitem{galvin}
{\sc Galvin, F.}
\newblock Letter to {R}ichard {L}aver, March 19, 1970.

\bibitem{MR376359}
{\sc Galvin, F., and Hajnal, A.}
\newblock Inequalities for cardinal powers.
\newblock {\em Ann. of Math. (2) 101\/} (1975), 491--498.

\bibitem{Sh:23}
{\sc Galvin, F., and Shelah, S.}
\newblock {Some counterexamples in the partition calculus}.
\newblock {\em J. Combinatorial Theory Ser. A 15\/} (1973), 167--174.

\bibitem{MR264585}
{\sc Hajnal, A., and Juh\'{a}sz, I.}
\newblock Discrete subspaces of topological spaces. {II}.
\newblock {\em Indag. Math. 31\/} (1969), 18--30.
\newblock Nederl. Akad. Wetensch. Proc. Ser. A {{\bf{7}}2}.

\bibitem{MR1994835}
{\sc Kanamori, A.}
\newblock {\em The higher infinite}, second~ed.
\newblock Springer Monographs in Mathematics. Springer-Verlag, Berlin, 2003.
\newblock Large cardinals in set theory from their beginnings.

\bibitem{MR2069032}
{\sc Larson, P.~B.}
\newblock {\em The stationary tower}, vol.~32 of {\em University Lecture Series}.
\newblock American Mathematical Society, Providence, RI, 2004.
\newblock Notes on a course by W. Hugh Woodin.

\bibitem{MR3728284}
{\sc Munkres, J.~R.}
\newblock {\em Topology}.
\newblock Prentice Hall, Inc., Upper Saddle River, NJ, 2000.
\newblock Second edition of [ MR0464128].

\bibitem{MR4190059}
{\sc Raghavan, D., and Todorcevic, S.}
\newblock Proof of a conjecture of {G}alvin.
\newblock {\em Forum Math. Pi 8\/} (2020), e15, 23.

\bibitem{MR4579382}
{\sc Raghavan, D., and Todorcevic, S.}
\newblock Galvin's problem in higher dimensions.
\newblock {\em Proc. Amer. Math. Soc. 151}, 7 (2023), 3103--3110.

\bibitem{MR1576401}
{\sc Ramsey, F.~P.}
\newblock On a {P}roblem of {F}ormal {L}ogic.
\newblock {\em Proc. London Math. Soc. (2) 30}, 4 (1929), 264--286.

\bibitem{safblog}
{\sc Rinot, A.}
\newblock {D}ushnik-{M}iller for singular cardinals (part 2).
\newblock \url{https://blog.assafrinot.com/?p=628}.
\newblock Accessed: 2024-05-14.

\bibitem{MR3271280}
{\sc Rinot, A.}
\newblock Chain conditions of products, and weakly compact cardinals.
\newblock {\em Bull. Symb. Log. 20}, 3 (2014), 293--314.

\bibitem{Sh:71}
{\sc Shelah, S.}
\newblock {A note on cardinal exponentiation}.
\newblock {\em J. Symbolic Logic 45}, 1 (1980), 56--66.

\bibitem{Sh:276}
{\sc Shelah, S.}
\newblock {Was Sierpi\'nski right? I}.
\newblock {\em Israel J. Math. 62}, 3 (1988), 355--380.

\bibitem{Sh:288}
{\sc Shelah, S.}
\newblock {Strong partition relations below the power set: consistency; was Sierpi\'nski right? II}.
\newblock In {\em {Sets, graphs and numbers (Budapest, 1991)}}, vol.~60 of {\em Colloq. Math. Soc. J\'anos Bolyai}. North-Holland, Amsterdam, 1992, pp.~637--668.
\newblock \href{https://arxiv.org/abs/math/9201244}{arXiv: math/9201244}.

\bibitem{Sh:g}
{\sc Shelah, S.}
\newblock {\em {Cardinal arithmetic}}, vol.~29 of {\em Oxford Logic Guides}.
\newblock The Clarendon Press, Oxford University Press, New York, 1994.

\bibitem{Sh:589}
{\sc Shelah, S.}
\newblock {Applications of PCF theory}.
\newblock {\em J. Symbolic Logic 65}, 4 (2000), 1624--1674.
\newblock \href{https://arxiv.org/abs/math/9804155}{arXiv: math/9804155}.

\bibitem{Sh:546}
{\sc Shelah, S.}
\newblock {Was Sierpi\'nski right? IV}.
\newblock {\em J. Symbolic Logic 65}, 3 (2000), 1031--1054.
\newblock \href{https://arxiv.org/abs/math/9712282}{arXiv: math/9712282}.

\bibitem{Sh:881}
{\sc Shelah, S.}
\newblock {The Erd\H{o}s-Rado arrow for singular cardinals}.
\newblock {\em Canad. Math. Bull. 52}, 1 (2009), 127--131.
\newblock \href{https://arxiv.org/abs/math/0605385}{arXiv: math/0605385}.

\bibitem{MR1556708}
{\sc Sierpi\'{n}ski, W.}
\newblock Sur un probl\`eme de la th\'{e}orie des relations.
\newblock {\em Ann. Scuola Norm. Super. Pisa Cl. Sci. (2) 2}, 3 (1933), 285--287.

\bibitem{Sierpinski1920}
{\sc Sierpiński, W.}
\newblock Sur une propriété topologique des ensembles dénombrables denses en soi.
\newblock {\em Fundamenta Mathematicae 1}, 1 (1920), 11--16.

\bibitem{MR2304318}
{\sc Todorcevic, S.}
\newblock Universally meager sets and principles of generic continuity and selection in {B}anach spaces.
\newblock {\em Adv. Math. 208}, 1 (2007), 274--298.

\bibitem{MR2355670}
{\sc Todorcevic, S.}
\newblock {\em Walks on ordinals and their characteristics}, vol.~263 of {\em Progress in Mathematics}.
\newblock Birkh\"{a}user Verlag, Basel, 2007.

\bibitem{MR2603812}
{\sc Todorcevic, S.}
\newblock {\em Introduction to {R}amsey spaces}, vol.~174 of {\em Annals of Mathematics Studies}.
\newblock Princeton University Press, Princeton, NJ, 2010.

\bibitem{todorcevicbasis}
{\sc Todorcevic, S.}
\newblock Basis problems and expansions, {III}.
\newblock {\em Fields Institute Seminar\/} (March 17, 2023).

\bibitem{MR908147}
{\sc Todor\v{c}evi\'{c}, S.}
\newblock Partitioning pairs of countable ordinals.
\newblock {\em Acta Math. 159}, 3-4 (1987), 261--294.

\bibitem{MR1297180}
{\sc Todor\v{c}evi\'{c}, S.}
\newblock Some partitions of three-dimensional combinatorial cubes.
\newblock {\em J. Combin. Theory Ser. A 68}, 2 (1994), 410--437.

\bibitem{todorcevicweiss}
{\sc Todor\v{c}evi\'{c}, S., and Weiss, W.}
\newblock Partitioning metric spaces, September 1995.

\bibitem{MR1083599}
{\sc Weiss, W.}
\newblock Partitioning topological spaces.
\newblock In {\em Mathematics of {R}amsey theory}, vol.~5 of {\em Algorithms Combin.} Springer, Berlin, 1990, pp.~154--171.

\bibitem{MR2723878}
{\sc Woodin, W.~H.}
\newblock {\em The axiom of determinacy, forcing axioms, and the nonstationary ideal}, revised~ed., vol.~1 of {\em De Gruyter Series in Logic and its Applications}.
\newblock Walter de Gruyter GmbH \& Co. KG, Berlin, 2010.

\end{thebibliography}
\end{document}